\documentclass[11pt,twoside]{amsart}
\usepackage{amssymb}

\usepackage[latin1]{inputenc}
\usepackage[T1]{fontenc}
\usepackage{mathtools}
\usepackage{graphicx}
\usepackage{tikz}
\usepackage{mathabx}
\usetikzlibrary{chains}
\usepackage[OT2,T1]{fontenc}

\PassOptionsToPackage{pdfusetitle,pagebackref,colorlinks}{hyperref}
\usepackage{bookmark}
\hypersetup{
  linkcolor={red!70!black},
  citecolor={green!70!black},
  urlcolor={blue!80!black}
}

\usepackage[latin1]{inputenc}
\usepackage{amsmath}
\usepackage{amsthm}

\usepackage[all]{xy}
\usepackage{hyperref}
\newtheorem{thm}{Theorem}[section]

\newtheorem{prop}[thm]{Proposition}

\newtheorem{lem}[thm]{Lemma}
\newtheorem{cor}[thm]{Corollary}
\newtheorem{conj}[thm]{Conjecture}
\numberwithin{equation}{section}

\theoremstyle{definition}

\newtheorem{remark}[thm]{Remark}

\DeclareFontFamily{U}{mathc}{}
\DeclareFontShape{U}{mathc}{m}{it}%
{<->s*[1.03] mathc10}{}
\DeclareMathAlphabet{\mathcal}{U}{mathc}{m}{it}

\newcommand{\im}{\operatorname{im}}
\newcommand{\Db}{{\rm D}^{\rm b}}

\newcommand{\Aut}{{\rm Aut}}

\newcommand{\Br}{{\rm Br}}

\newcommand{\per}{{\rm per}}
\newcommand{\ind}{{\rm ind}}
\newcommand{\CH}{{\rm CH}}
\newcommand{\NS}{{\rm NS}}
\newcommand{\Pic}{{\rm Pic}}

\newcommand{\rk}{{\rm rk}}

\newcommand{\coker}{{\rm coker}}

\newcommand{\cal}{\mathcal}
\newcommand{\ka}{{\cal A}}
\newcommand{\kb}{{\cal B}}
\newcommand{\kc}{{\cal C}}

\newcommand{\ke}{{\cal E}}

\newcommand{\kk}{{\cal K}}
\newcommand{\kl}{{\cal L}}

\newcommand{\km}{{\cal M}}
\newcommand{\ko}{{\cal O}}
\newcommand{\kp}{{\cal P}}

\newcommand{\ky}{{\cal Y}}

\newcommand{\GG}{\mathbb{G}}

\newcommand{\ZZ}{\mathbb{Z}}
\newcommand{\QQ}{\mathbb{Q}}

\newcommand{\PP}{\mathbb{P}}

\newcommand{\Gm}{\mathbb{G}_m}

\DeclareSymbolFont{cyrletters}{OT2}{wncyr}{m}{n}
\DeclareMathSymbol{\Sha}{\mathalpha}{cyrletters}{"58}

\renewcommand{\to}{\xymatrix@1@=15pt{\ar[r]&}}
\newcommand{\lto}{\xymatrix@1@=15pt{&\ar[l]}}
\renewcommand{\rightarrow}{\xymatrix@1@=15pt{\ar[r]&}}
\renewcommand{\mapsto}{\xymatrix@1@=15pt{\ar@{|->}[r]&}}
\newcommand{\mapslto}{\xymatrix@1@=15pt{&\ar@{|->}[l]&}}
\renewcommand{\twoheadrightarrow}{\xymatrix@1@=18pt{\ar@{->>}[r]&}}

\renewcommand{\hookrightarrow}{\xymatrix@1@=15pt{\ar@{^(->}[r]&}}
\newcommand{\hook}{\xymatrix@1@=15pt{\ar@{^(->}[r]&}}
\newcommand{\twoo}{\twoheadrightarrow }
\newcommand{\congpf}{\xymatrix@1@=15pt{\ar[r]^-\sim&}}
\renewcommand{\cong}{\simeq}

\newcommand{\TBC}[1]{}
\newcommand{\MD}[1]{}
\newcommand{\EV}[1]{}

\makeatletter
\def\blfootnote{\xdef\@thefnmark{}\@footnotetext}
\makeatother

\begin{document}

\title[The period-index problem for hyperk\"ahler manifolds]{The period-index problem for hyperk\"ahler manifolds}

\author[D.\ Huybrechts]{Daniel Huybrechts}

\address{Mathematisches Institut \& Hausdorff Center for Mathematics,
Universit{\"a}t Bonn, Endenicher Allee 60, 53115 Bonn, Germany}
\email{huybrech@math.uni-bonn.de}

\begin{abstract}  \vspace{-2mm} We conjecture that every unramified Brauer class
$\alpha\in \Br(X)$ on a projective hyperk\"ahler manifold $X$ satisfies $\ind(\alpha)\mid\per(\alpha)^{\dim(X)/2}$.
We provide evidence for this conjecture by proving it for two large classes of projective
hyperk\"ahler manifolds: For projective hyperk\"ahler manifolds admitting a Lagrangian
fibration and for Hilbert schemes of K3 surfaces.\end{abstract}

\maketitle
\blfootnote{The author is supported by the ERC Synergy Grant HyperK (ID 854361).}

Two numerical invariants are associated with every Brauer class $\alpha\in\Br(X)$ on a variety $X$: Its period $\per(\alpha)$, which is its order as an element of the abelian group $\Br(X)$,
and its index $\ind(\alpha)$. To define the latter recall that every class $\alpha$ can be represented
by a Brauer--Severi variety $P\to X$ with fibres $\PP^{m-1}$. Then $\ind(\alpha)$  can be defined as the gcd of all occurring $m$. \smallskip

It is not difficult to see that $\per(\alpha)$ divides $\ind(\alpha)$ and that
the two invariants have the same prime factors, cf.\ \cite[(5.3)]{Artin} or \cite[Prop.\ 4.5.13]{GS}. In particular, $\ind(\alpha)$ divides some power of $\per(\alpha)$. However, whether this power can be chosen independently of $\alpha$ and to what extent it
depends on $X$, remains largely open. The following conjecture has been put forward by Colliot-Th\'el\`ene
\cite{CTBr}.

\begin{conj}\label{conj1}
If $X$ is a factorial projective variety over an algebraically closed field, then 
$$\ind(\alpha)\mid\per(\alpha)^{\dim(X)-1}$$
for any Brauer class $\alpha\in\Br(X)$.
\end{conj}

The conjecture holds true for curves, as they have trivial Brauer groups, and also for surfaces. The latter is a
result of de Jong \cite{dJ} in characteristic zero and of de Jong--Starr \cite{dJS} and Lieblich \cite{LiebAnn}
in positive characteristic. It has recently also been verified by Hotchkiss and Perry \cite{HoPe} for abelian varieties of dimension three.  The existence of a universal power depending only on 
the geometry of $X$, namely $\per(\alpha)^{2g}$ with $g$ the genus of a very ample complete
intersection curve, has recently been established in joint work with Mattei \cite{HuyMa} and
before  that by de Jong and Perry \cite{dJP} under the hypothesis that the standard Lefschetz conjecture holds in degree two.
In the appendix to a paper of Colliot-Th\'el\`ene \cite{CTGab}, Gabber showed that the bound is optimal. More precisely, in any dimension $d$, there exists
a product of curves $X=C_1\times\cdots\times C_d$ together with a Brauer class $\alpha\in \Br(X)$ such that $\ind(\alpha)=\per(\alpha)^{d-1}$.
However, for particular classes of varieties the exponent might be smaller, and this is the topic of the present article. 
\medskip

Apart from abelian varieties there is essentially only one other class of varieties with trivial canonical bundle and  non-trivial Brauer group: hyperk\"ahler manifolds. In many respects, hyperk\"ahler manifolds do not behave like general complex projective varieties, e.g.\ their geometry is largely determined by their
Hodge structure of weight two and not their middle cohomology. We expect hyperk\"ahler manifolds to be particular also with respect to the period-index problem in the following sense.

\begin{conj}\label{conj2}
Let $X$ be a projective hyperk\"ahler manifold and $\alpha\in \Br(X)$. Then
$$\ind(\alpha)\mid\per(\alpha)^{\dim(X)/2}.$$
\end{conj}

Note that for two-dimensional hyperk\"ahler manifolds, i.e.\ for K3 surfaces, the conjecture simply
says $\ind(\alpha)=\per(\alpha)$ which is known by de Jong's result \cite{dJ}. In fact, it is also true for non-projective K3 surfaces as proved in joint work with Schr\"oer \cite{HuySchr}. 
Recall that for surfaces the index is indeed realised by an Azumaya algebra on the whole surface.
This is a consequence of de Jong's proof \cite{dJ} or, in the case of K3 surface, of \cite{HuySchr}, but it also follows from results of Auslander and Goldman, cf.\ \cite{AnW}.\MD{See vdBergh's account of de Jong. Reflexive sheaves are locally free!}
\smallskip

The goal of this article is to provide evidence for Conjecture \ref{conj2} in higher dimensions. We essentially prove the conjecture for hyperk\"ahler manifolds that admit a Lagrangian fibration, 
which according to the SYZ conjecture form a countable and dense union of codimension one subvarieties in the moduli space of all (projective) hyperk\"ahler manifolds, cf.\  \cite[Prop.\ 7.1.3]{HuyK3}. Density has in fact been established for all known deformation types, see \cite{DIKM} for a survey and references.
\smallskip

The first main result of this article is the following:

\begin{thm}\label{thm1}
Let $X$ be a projective hyperk\"ahler manifold that admits a  Lagrangian fibration. Then there exists an integer $N_X$ such that
$$\ind(\alpha)\mid\per(\alpha)^{\dim(X)/2}$$
for all
$\alpha\in\Br(X)$ with $\per(\alpha)$ coprime to $N_X$.
\end{thm}

Ideally, of course, one would like to prove $N_X=1$. Our methods do not currently
give this, but with some effort
we manage to produce a concrete $N_X$ that depends in an interesting way on the geometry
of $X$ and the assumed Lagrangian fibration $h\colon X\to \PP^n$. The answer will be of the form
\begin{equation}\label{eqn:NX}
N_X=\deg(\kl)\cdot\deg (\tilde\PP/\PP)\cdot |H^3(X,\ZZ)_{\text{ tors}}|\cdot \ind (T(X)_0),
\end{equation}
where $\deg(\kl)=\chi(X_t,\kl_t)$ is the minimal fibre degree of a relative polarisation $\kl$, $\deg(\tilde\PP/\PP)$ is the degree of a generically finite Lagrangian multi-section $\tilde\PP\subset X$ of the Lagrangian fibration, which always exist, and $\ind(T(X)_0)$ is the index of $T(X)_0\coloneqq\ker(T(X)\to H^0(\PP^n,R^2h_\ast\ZZ))$
viewed as a finite index sub-lattice of $H^2(X,\ZZ)/{\rm NS}(X)$. We refer to \S\! \ref{sec:proof} for a mored detailed discussion of $N_X$.\smallskip

We also point out that the weaker bound $\ind(\alpha)\mid\per(\alpha)^{\dim(X)}$ is
much easier to prove, see Remark \ref{rem:2gbis}. Moreover, it suffices to assume
that $\per(\alpha)$ is coprime to $\deg(\tilde \PP/\PP)$ and the order of $H^3(X,\ZZ)_{\text{ tors}}$.
\medskip

The second main result is concerned with Hilbert schemes of K3 surfaces, one of the principal
examples of hyperk\"ahler manifolds, and uses the identification of the Brauer group of a K3 surface $S$ and of its Hilbert scheme $S^{[n]}$ provided by Hodge theory. As the previous result,
the next theorem proves Conjecture \ref{conj2} for a dense set of all hyperk\"ahler manifolds
in this particular deformation class,  see \cite{MarkMeh}.

\begin{thm}\label{thm2}
Let $X=S^{[n]}$ be the Hilbert scheme of a K3 surface and $\alpha\in \Br(X)$. Then
$$\ind(\alpha)\mid\per(\alpha)^n.$$
\end{thm}

The techniques we will use to prove Theorem \ref{thm2} can also be used to prove Theorem \ref{thm1} for certain Lagrangian fibred hyperk\"ahler manifolds of ${\rm K3}^{[n]}$-type without any assumption on $\per(\alpha)$, so with $N_X=1$, cf.\ \S\! \ref{sec:ModuliK3curves} \& \S\! \ref{sec:HilbLag}. It would be interesting to generalise Theorem \ref{thm2} to arbitrary smooth projective moduli spaces of sheaves on K3 surfaces. \smallskip

\EV{Note that at this point  Conjecture \ref{conj2} has not been proved for all hyperk\"ahler manifolds of any given deformation type. Also it remains unclear if it is reasonable to expect similar results for arbitrary
ramified Brauer classes $\alpha\in\Br(K(X))$, see \S\! \ref{sec:ramBr}.}

\medskip

\noindent
{\bf Conventions and background:}  A projective hyperk\"ahler manifold $X$ is a smooth complex
projective variety that is simply connected and for which $H^0(X,\Omega^2_X)$ is spanned
by an everywhere non-degenerate two-form. They are to abelian varieties what K3 surfaces are to abelian surfaces. They are also sometimes called projective irreducible
holomorphic symplectic varieties, see \cite{GHJ} for definitions and basic facts.
A Lagrangian fibration of a compact hyperk\"ahler manifold is a connected morphism $h\colon X\to \PP$ 
onto a normal variety $\PP$ of dimension $\dim(\PP)=\dim(X)/2$. If $\PP$ is smooth, which is equivalent to $h$ being flat, then $\PP\cong\PP^n$, cf.\ \cite{BS,Hwang,LiTo}.
\smallskip

Only two series of deformations types are known: Deformations of the Hilbert scheme $S^{[n]}$ of a
K3 surface $S$ and deformations of generalised Kummer varieties $K_n(A)$ associated with an abelian surface $A$. In addition, there are two sporadic types, OG6 and OG10, in dimension six and ten \cite{OG1,OG2}.
In recent years, singular examples have been constructed, typically by taking quotients, and it seems reasonable to expect Conjecture \ref{conj2} to hold for those as well.
\smallskip

Brauer groups of schemes have been introduced by Grothendieck and his three papers on the subject \cite{Brauer} remain highly influential and a standard reference. A modern account is provided by the recent monograph \cite{CTS} by Colliot-Th\'el\`ene and Skorobogatov.\smallskip
\medskip

\noindent
{\bf Main techniques:}
Theorem \ref{thm1} is proved by exploiting twisted Picard varieties as considered in \cite{HuyMa}
and techniques going back to Ogg \cite{Ogg} and \v{S}afarevi\v{c} \cite{Sha} for elliptic curves, 
generalised by Bhatt \cite{KLB} and Clark \cite{Clark} to  abelian varieties. Theorem \ref{thm2} 
is reduced to the result for surfaces, cf.\ \cite{HuySchr,dJ}, by means of exterior products of Azumaya algebra.\smallskip

It is tempting to approach the period-index problem for arbitrary
projective hyperk\"ahler manifolds by applying deformation theory,
starting with the results for the two dense sets covered by the 
techniques in this paper, cf.\ Remark \ref{rem:notfeas}.\smallskip

Some of the techniques developed in this paper also apply to more general situations.
All of \S\! \ref{sec:PIAbsch} is concerned with general abelian fibrations, \S\! \ref{sec:compactify}
applies to modular compactifications of dual fibrations quite broadly, and \S\! \ref{sec:Improv}
outlines an argument how to improve the exponent $2g(C)$ in \cite{HuyMa} to just $g(C)$
for arbitrary projective varieties.

\medskip

\noindent
{\bf Acknowledgements:} Early versions of the results in this paper have been presented
at a workshop in Cetraro and in the S\'eminaire de G\'eom\'etrie Alg\'ebrique
 in Paris. I wish to thank the participants, especially Olivier Benoist, Bruno Kahn,
 and Claire Voisin, for useful comments and challenging questions. I am particularly
 grateful to Dominique Mattei for the collaboration leading to \cite{HuyMa,HuyMa2} and
 his comments on the first version of the present paper. Thanks also to James Hotchkiss
 for pointing out the reference \cite{CTGab}.

\section{Period-index for abelian schemes}\label{sec:PIAbsch}

Parts of the discussion in this section is a geometric remake of the more arithmetic ones by Bhatt  \cite{KLB},
 Clark \cite{Clark}, and the classical ones by Ogg \cite{Ogg} and \v{S}afarevi\v{c} \cite{Sha}.
Our slightly more geometric point of view will be crucial for the applications to hyperk\"ahler varieties in the next section.
\smallskip

\subsection{Twisting the dual abelian scheme} In the following, $B$ is a smooth and quasi-projective variety over an algebraically closed field $k$ of characteristic zero.\TBC{Do we need this a this point? No!} 

Assume $f\colon A\to B$ is an abelian scheme of relative dimension $g$ and denote by
$$\check f\colon\check{A}\coloneqq\Pic^0(A/B)\to B$$  its dual.  In most of the discussion below it is enough to assume that 
the geometric fibres $A_t$ of $f$ are isomorphic to abelian varieties, only the abelian scheme structure
of $\check A\to B$ is of real importance. The zero section of $A$ is solely used for the existence of various Poincar\'e bundles
and we will point this out explicitly when the time comes.\smallskip

With this situation, the following \'etale sheaves are naturally associated:
\begin{enumerate}\setlength\itemsep{0.3em}
\item[(i)] The sheaf $\underline{\check{A}}$  of sections of $\check{f}$;
\item[(ii)] The sheaf $\underline{\Aut}(\check{A}/B)$ of relative automorphisms of the smooth projective morphism $\check{f}\colon \check{A}\to B$  together with 
its identity component $\underline{\Aut}^0(\check{A}/B)\subset \underline{\Aut}(\check{A}/B)$;
\item[(iii)] The sheaf $R^1f_\ast\Gm$ together with its identity component $(R^1f_\ast\Gm)^0\subset R^1f_\ast\Gm$.
\end{enumerate}
\smallskip
Then there are natural isomorphisms
\begin{equation}\label{eqn:WC}
(R^1f_\ast\Gm)^0\cong\underline{\check{A}}\cong \underline{\Aut}^0(\check{A}/B).
\end{equation}
In particular, the Weil--Ch\^atelet group $H^1(B,\underline{\check{A}})$ of the abelian scheme $\check{A}\to B$ is identified
with $H^1(B,(R^1f_\ast\Gm)^0)$.
\smallskip

As soon as $g>1$, the sheaf $R^2f_\ast\Gm$ is not trivial and so  the restriction map $$\Br(A)\cong H^2(A,\Gm)\to H^0(B, R^2f_\ast\Gm),$$ which is $E^2\to E_2^{0,2}$ in the Leray spectral sequence, may not be trivial. We denote its kernel by
$$\Br_1(A)\coloneqq \ker\left(\Br(A)\cong H^2(A,\Gm)\to H^0(B, R^2f_\ast\Gm)\right).$$
It can also be seen as the subgroup of all geometrically trivial Brauer classes on the generic fibre, i.e.\  $\Br_1(A)=\ker(\Br(A)\,\hookrightarrow\Br(A_\eta)\to \Br(A_{\bar\eta}))$.

From the Leray spectral sequence we obtain a natural map 
$$\Br_1(A)\to H^1(B, R^1f_\ast\Gm)$$ and we denote by $\Br_1(A)^0$ the pre-image of
$H^1(B, (R^1f_\ast\Gm)^0)$ which then comes with a natural map 
$\Br_1(A)^0\to H^1(B, (R^1f_\ast\Gm)^0)$ to the Weil--Ch\^atelet group 
of the abelian scheme $\check A \to B$. Note for later use that for the definition
of $\Br_1(A)^0$ we do not need to assume that $A\to B$ has a section.\smallskip

Thus, with any class $\alpha\in\Br_1(A)^0$,
there is naturally associated a twist  $A_\alpha\to B$ of $\check{A}\to B$:
$$\alpha \in \Br_1(A)^0 \leadsto A_\alpha\to B.$$
To make this more explicit, realise the image of $\alpha$ in the Weil--Ch\^atelet group as a cocycle $$\{\alpha_{ij}\}\in
H^1(B,(R^1f_\ast\Gm)^0)\cong H^1(B, \underline{\Aut}^0(\check{A}/B))$$ with respect to a certain \'etale covering $\bigcup U_i\to B$. Then the open subsets $\check{A}_{U_{ij}}\coloneq\check{f}^{-1}(U_{ij})$ are
glued via the isomorphims $\alpha_{ij}\colon \check{A}_{U_{ij}}\cong\check{A}_{U_{ji}}$ to form a scheme $A_\alpha\to B$, cf.\ \cite[VIII.\ Cor.\ 7.7]{SGA1}.\smallskip

Here are a few standard facts:\smallskip

\begin{enumerate}\setlength\itemsep{0.3em}
\item[(o)] The fibration $A_\alpha\to B$ is a torsor for the abelian scheme $\check{A}\to B$.
\item[(i)] The isomorphism type of $A_\alpha$ together with the fibration $A_\alpha\to B$  only depends on the
image of $\alpha$ in the Weil--Ch\^atelet group $H^1(B,(R^1f_\ast\Gm)^0)$.
\item[(ii)] The torsor $A_\alpha$ is trivial, i.e.\  $\check{A}\cong A_\alpha$ over $B$, if $\alpha$ is trivial. Also, $A_\alpha$ is trivial if and only if the projection $A_\alpha\to B$ has a section.
\item[(iii)] The construction is compatible with the group structure of the Weil--Ch\^atelet group,
i.e.\ for two classes $\alpha_1$ and $\alpha_2$ there exists a morphism
$A_{\alpha_1}\times_B A_{\alpha_2}\to A_{\alpha_1\alpha_2}$ compatible with the actions of
$\check{A}\to B$.
\end{enumerate}

\begin{remark}
Note that for $g=1$, i.e.\  $A\to B$ is an elliptic fibration (a genus one fibration without a section is enough), one has $$\Br_1(A)^0=\Br_1(A)=\Br(A).$$ 
Indeed, $R^2f_\ast\Gm$ is trivial and $(R^1f_\ast\Gm)^0 \subset
 R^1f_\ast\Gm$ admits a split.
 \end{remark}

\subsection{Multi-sections}\label{sec:Multi}
If now $\alpha\in \Br_1(A)^0$ is of  finite order, say $r$, then we get a morphism from the $r$-fold fibre product
to $\check{A}$:
\begin{equation}\label{eqn:group1}
A_\alpha\times_B\cdots\times_BA_\alpha\to A_{\alpha^r}\cong \check{A},
\end{equation}
which on each fibre $ (A_{\alpha})_t\cong \check{A}_t$ is  described (up to translation) by the group structure on the abelian variety $\check{A}_t$. In particular, the composition of (\ref{eqn:group1}) with the diagonal embedding
$A_\alpha\,\hookrightarrow A_\alpha\times_B\cdots\times_BA_\alpha$ defines an \'etale morphism
$$\varphi\colon A_\alpha\twoheadrightarrow \check{A}$$
of degree $r^{2g}$. Taking the pre-image of the zero section $B\subset \check{A}$ of $\check f\colon \check A\to B$ then leads to 
a multi-section 
$$\varphi^{-1}(B)\subset A_\alpha\to B$$ of degree $r^{2g}$.\footnote{The idea to consider the
pre-image in this context seems a standard trick in the area that was first brought to my attention independently by B.\ Antieau, A.\ Auel, and M.\ Lieblich.}

\begin{remark}
Another way to phrase the discussion so far goes as follows, see also the discussion further below
for similar arguments: If
$\alpha\in H^1(B,(R^1f_\ast\Gm)^0)=H^1(B,\underline{\check A})$ is of order $r$, then it 
comes from a class in $H^1(B,R^1f_\ast\mu_r)=H^1(B,\underline{\check A}[r])$ and is, therefore,
represented by a cocyle of relative automorphisms acting by translation by $r$-torsion
points. Since translation by $r$-torsion points preserves the set of $r$-torsion points in the fibres,
this produces a  multi-section of $A_\alpha\to B$ of degree $r^{2g}$. 

If one would like to argue in a similar fashion to produce multi-sections of smaller degree, then
a relative cycle of $\check A\to B$ 
 would be needed that is preserved by some cocycle representing $\alpha$. There is no obvious choice of such a relative cycle even in the case of an elliptic surface.
\end{remark}

So, at this point in the discussion, the techniques seem unlikely to produce any multi-section of smaller degree.
However, if $A\to B$ is an elliptic surface $S\to C$, i.e.\ $g=1$ and $\dim(B)=1$, then already  Ogg \cite{Ogg}, \v{S}afarevi\v{c} \cite{Sha} and later again Lichtenbaum \cite{Li} observed that a multi-section of smaller degree $r$ instead of $r^2$ does exist. More precisely, they established the following theorem. A proof of it will be sketched  below.

\begin{thm}[Ogg, \v{S}afarevi\v{c}]\label{thm:OS}
Assume $S\to C$ is a smooth elliptic fibration over a quasi-projective curve $C$. Then the $\alpha$-twist  $S_\alpha\to C$ of the dual  elliptic fibration $\check{S}\to C$ associated with a class $\alpha\in\Br(S)$ of order $r$ admits a multi-section of degree $r$.
\end{thm}

For a smooth elliptic fibration $S\to C$  (over a curve), producing small degree multi-sections of a twist
$S_\alpha\to C$ is essentially equivalent to producing line bundles of small fibre degree on $S_\alpha$.
However, for $g>1$, so abelian schemes $A\to B$ of higher relative dimensions, this is no longer
the case and, in fact, for the application to the period-index problem the line bundle approach
gives better results, cf.\ \S\! \ref{sec:ApplPI} for more details. As an illustration, we will briefly
review the line bundle approach first in the case of elliptic surfaces.


\subsection{Gluing line bundles on twists of elliptic surfaces}\label{sec:GluingCurve}
Assume $S\to C$ is a smooth elliptic fibration over a quasi-projective curve and let $L\coloneqq\ko(C)\in \Pic(\check{S})$ be  the line bundle associated with the zero section  $C\subset\check{S}$ of the dual elliptic fibration $\check{S}\to C$. We claim that  for any class $\alpha\in \Br(S)$ of order $r$ the associated element in the Weil--Ch\^atelet group can be represented by a cocycle that also glues $L^r$.

The short exact sequence 
$$\xymatrix{0\ar[r]&\check{S}[r]\ar[r]&\check{S}\ar[r]^-r&\check{S}\ar[r]& 0}$$of group schemes over $C$
induces a long exact sequence
$$\xymatrix@R=5pt{\cdots\ar[r]&H^1(C,\underline{\check{S}}[r])\ar[r]&H^1(C,\underline{\check{S}})\ar[r]^-r&H^1(C,\underline{\check{S}})\ar[r]&\cdots.\\
&\alpha'\ar@{|->}[r]&\alpha&&}$$
Since $r$ is the order of $\alpha$, we can lift $\alpha$ to a class  $\alpha'\in H^1(C,\underline{\check{S}}[r])$,
which can be interpreted as a torsor for the finite group scheme $\check{S}[r]\to C$ of $r$-torsion points.
\smallskip

Let us first look at the case of a single elliptic curve $E$. 
Clearly, for a line bundle $L$ on the elliptic curve $E$, one has $t_a^\ast L^r\cong L^r$ for every $r$-torsion point $a\in E$, i.e.\ $E[r]\subset K(L^r)$, where $K(L^r)$ is the kernel of the homomorphism $\varphi_{L^r}\colon E\to \check{E}$, $a\mapsto t_a^\ast L^r\otimes L^{-r}$.

The theta group $\tilde K(L^r)$ of $L^r$, cf.\ \cite[\S\! 1]{MumfordTheta}, is the extension 
\begin{equation}\label{eqn:theta}
\xymatrix{0\ar[r]&k^\ast\ar[r]&{\tilde K(L^r)}\ar[r]&K(L^r)\ar[r]&0}
\end{equation}
of all automorphisms of $L^r$ over $E$ acting by translation on $E$. Restricting to the subgroup $E[r]=K(L^r)[r]\subset K(L^r)$
gives 
\begin{equation}\label{eqn:theta2}
\xymatrix{0\ar[r]&k^\ast\ar[r]&\Aut(E,L^r)[r]\ar[r]&E[r]\ar[r]&0,}
\end{equation}
where $\Aut(E,L^r)[r]\subset \Aut(E,L^r)$ denotes the group of automorphisms of $(E,L^r)$  acting by translation by
$r$-torsion points on $E$. 
\smallskip

 The relative version of (\ref{eqn:theta2}) for $\check{S}\to C$ is the short exact sequence
 of sheaves
 \begin{equation}\label{eqn:Heisenberg}
\xymatrix{0\ar[r]&\Gm\ar[r]&\underline{\Aut}((\check{S},L^r)/C)[r]\ar[r]&\underline{\check{S}}[r]\ar[r]&0}
\end{equation}
with the induced exact cohomology sequence
$$\xymatrix@R=5pt{H^1(C, \underline{\Aut}((\check{S},L^r)/C)[r])\ar[r]&H^1(C,\underline{\check{S}}[r])\ar[r]^-\delta&H^2(C,\Gm).\\
\tilde\alpha\ar@{|->}[r]&\alpha'&}$$
The class $\alpha'\in H^1(C,\underline{\check{S}}[r])$ can be lifted further
to a class $\tilde\alpha\in  H^1(C, \underline{\Aut}((\check{S},L^r)/C)[r])$, for $\delta(\alpha')\in\Br(C)\cong H^2(C,\Gm)=0$ by Tsen's theorem.
\smallskip

The upshot of this discussion is that the twist $S_\alpha\to C$ of the dual fibration
$\check{S}\to C$ (which is of course isomorphic to the original one $S\to C$) comes with a line bundle $M_\alpha\in \Pic(S_\alpha)$ of  degree $r$ on the fibres. After twisting by a sufficiently positive line bundle on $C$ we may assume that $M_\alpha$ has a non-trivial global section which eventually produces a multi-section of degree $r$ on
$S_\alpha\to C$. This concludes the proof of the  Ogg--\v{S}afarevi\v{c} Theorem \ref{thm:OS}.


\subsection{Gluing line bundles on twists of abelian schemes}\label{sec:glueAS}
As in the classical Ogg--\v{S}afarevi\v{c} theory, it is not evident how to find multi-sections
of $A_\alpha\to B$ of degree smaller than $r^{2g}$. Inspired by the discussion in the case
of elliptic surfaces, we will glue line bundles instead. Very similar arguments have been presented
by Abasheva \cite[\S ~3.0.3]{Abash}.
\smallskip

So, let $f\colon A\to B$ be an abelian scheme of relative dimension $g$. For what follows in this
section, we only need that the geometric fibres of $f$ are isomorphic to abelian varieties, the existence of a zero section of $f$
is not needed.

Assume  $L$ is a line bundle on the dual fibration $\check{A}\to B$, later it will be assumed 
relatively ample or at least of non-zero fibre degree, and let  $\alpha\in H^1(B,(R^1f_\ast\Gm)^0)$ be of order $r$. In the proof of the next proposition,
we will revisit the construction of a natural obstruction class, cf.\
\cite{Clark,KLB,Ogg,Sha} $$o(\alpha,L^r)\in H^2(B,\Gm),$$ 
for which the following proposition holds.

\begin{prop}\label{prop:lbabscheme}
If $o(\alpha,L^r)\in H^2(B,\Gm)$ is trivial, then there exists
a line bundle $M_\alpha$ on $A_\alpha$ satisfying fibrewise the equality
$\chi(A_{\alpha t},M_{\alpha t})=r^g\cdot\chi(A_t,L_t)$.
\end{prop}

\begin{proof}
This really is a copy of the arguments in \S\!  \ref{sec:GluingCurve}. 

Firstly, one lifts $\alpha\in H^1(B,\underline{\check{A}})$ to a class $\alpha'$ in $H^1(B,\underline{\check{A}}[r])$, which only needs that the order of $\alpha$ divides $r$. Clearly, the order of the lifted class $\alpha'$ in $H^1(B,\underline{\check{A}}[r])$ again divides $r$.

Secondly, one would like to lift $\alpha'$ further under $$
\xymatrix{H^1(B,\underline\Aut((\check A,L^r)/B)[r])\ar[r]& H^1(B,\underline{\check A}[r])\,,}$$
but this is obstructed by the image of $\alpha'$ under the boundary map\TBC{Think this through. It seems $\alpha'$ may not be unique and then $\delta(\alpha')$ is not?}
\begin{equation}\label{eqn:boundary}
\xymatrix{\delta\colon H^1(B,\underline{\check A}[r])\ar[r]& H^2(B,\Gm)}
\end{equation}
of the short exact sequence, the analogue of (\ref{eqn:Heisenberg}):
\begin{equation}\label{eqn:sesAV}
\xymatrix{0\ar[r]&\Gm\ar[r]&\underline\Aut((\check A,L^r)/B)[r]\ar[r]&\underline{\check A}[r]\ar[r]&0\,.}
\end{equation}
Here, we are using again that $L^r$ is fixed by translation by $r$-torsion points of $\check A$, i.e.\ the group scheme of $r$-torsion points
$\check A[r]$ is contained in the kernel of the homomorphism $\varphi_{L^r}\colon \check A\to  A$.\smallskip

By construction, the obstruction class $o(\alpha,L^r)\coloneqq\delta(\alpha')$, 
has the property that if it vanishes, then the restrictions of $L^r$ to the open subsets $\check{A}_{U_{ij}}$
can be glued  to a line bundle $M_\alpha$ on $A_\alpha$. As the restriction of $M_\alpha$ to each
fibre is then isomorphic to the restriction of $L^r$, one obtains the claimed equality.
\end{proof}

\begin{remark}
If we think of ${\rm c}_1(L)^g$ as a relative cycle on $\check{A}\to B$ of dimension zero and degree $\ell$, then the proposition can be read as saying that, once $o(\alpha,L^r)$ vanishes, $A_\alpha\to B$ admits the relative cycle ${\rm c}_1(M_\alpha)^g$ of degree $r^g\cdot \ell$. 
Note, however, that $\chi(L)=\ell/g!$ so that this approach usually leads to $\ind(\alpha)\mid g!\cdot\per(\alpha)^g$, which introduces further coprimality restrictions on $\per(\alpha)$. Also note that, unlike the case of elliptic surfaces, it is a little trickier to actually produce a multi-section of this degree, without assuming $L$ to have some  further positivity properties.
\end{remark}

\subsection{Application to the period-index problem}\label{sec:ApplPI}
We will now link the discussion so far to the period-index problem
and as a warm up we start again with the case of a smooth elliptic fibration $S\to C$ over a curve.\smallskip

For $\alpha\in\Br(S)$, the twist $S_\alpha$ can be reinterpreted purely in terms of $\alpha\in\Br(S)$ as a
certain moduli space $\Pic^0_\alpha(S/C)$ of $\alpha$-twisted sheaves on $S$ concentrated on the fibres
of $S\to C$, cf.\ \cite[Prop.\ 3.5]{HuyMa2}. To be more precise, the restriction of the class $\alpha\in H^1(C,R^1f_\ast\Gm)$ defines a torsor $\Pic_\alpha(S/C)\to C$ for the group scheme $\Pic(S/C)=\bigsqcup \Pic^d(S/C)\to C$. This torsor $\Pic_\alpha(S/C)$ is also the disjoint union of countably many components: 
 $\Pic_\alpha(S/C)=\bigsqcup\Pic_\alpha^i(S/C)$ with each component being a $\Pic^0(S/C)$-torsor. 
 Moreover, with the right indexing, tensor product defines a morphism $\Pic^1(S/C)\times \Pic^i_\alpha(S/C)\to \Pic^{i+1}_\alpha(S/C)$. As
$S\to C$ has a section, all $\Pic^0(S/C)$-torsors $\Pic^d(S/C)$ are isomorphic and hence
all $\Pic^i_\alpha(S/C)$ are isomorphic. See \cite{HuyMa2} for a detailed discussion how the components
$\Pic^i_\alpha(S/C)$ can be effectively enumerated.
In any case, for our discussion, the ambiguity in the choice of $i$ can be safely ignored and we will simply write $\Pic^0_\alpha(S/C)$ for one fixed component.
\smallskip

The Brauer class $\alpha\in \Br(S)$ restricted to the zero section $C\subset S$ of $f\colon S\to C$  is trivial for dimension reasons.
This rigidifies the twisted Picard functor and, as a consequence, $\Pic^0_\alpha(C)$ is a fine moduli space,
i.e.\ there exists a Poincar\'e bundle $$\kp_\alpha\to S_\alpha\times_CS,$$
which by definition is $(1,\alpha)$-twisted. Here, $(1,\alpha)\in \Br(S_\alpha\times_CS)$ is obtained by pulling-back $\alpha$ from the second factor. Note that $\kp_\alpha$ is unique only up to tensoring with line bundles in $S_\alpha$ and we can and will assume that via $\kp_\alpha$ the surface $S$ parametrises line bundles of degree zero on the fibres of $S_\alpha$ inducing an isomorphism $S\cong\Pic^0(S_\alpha/C)$.

The universal property of the Poincar\'e bundle then implies that taken as the kernel
of a Fourier--Mukai functor, $\kp_\alpha$ induces an equivalence $$\Phi_{\kp_\alpha}\colon \Db(S_\alpha)\congpf \Db(S,\alpha).$$
In particular,  if, as in \S\! \ref{sec:GluingCurve}, $M_\alpha\in\Pic(S_\alpha)$ denotes the line bundle that is obtained from gluing $L^r$ on $\check S$, where $L=\ko(C)\in \Pic(\check{S})$ is associated
with the zero section $C\subset \check{S}$, then $\Phi_{\kp_\alpha}(M_\alpha)={\rm pr}_{2\ast}(\kp_\alpha\otimes{\rm pr}_{1}^\ast M_\alpha)$ is an $\alpha$-twisted locally free sheaf of rank $r=\chi(S_{\alpha t},M_{\alpha t})$ on $S$.

Alternatively, we can choose a global section of $M_\alpha$ and use its zero set $Z_\alpha\subset S_\alpha$,
which is a multi-section of degree $r$. Then, the restriction of $\kp_\alpha$ to $Z_\alpha\times_CS$ can be 
viewed more directly as an $\alpha$-twisted locally free sheaf  of rank $r$ on $S$. This proves the following
classical result.

\begin{cor}[Ogg, \v{S}afarevi\v{c}]
For any $\alpha\in \Br(S)$, one has $\ind(\alpha)=\per(\alpha)$.\qed
\end{cor}
Observe that the discussion also shows that  $S_\alpha\to C$ cannot have a multi-section of order strictly
smaller than the order of $\alpha$.

\medskip

Let us now pass to abelian schemes. As for the case of elliptic surfaces, the twist $A_\alpha$ for $\alpha\in \Br_1(A)^0$ can be viewed as a component  $\Pic_\alpha^0(A/B)\to B$  of the
twisted Picard variety of $\alpha$-twisted line bundles on the fibres of $A\to B$.

As in the classical theory of Picard schemes \cite{BLR}, in order to ensure that $\Pic_\alpha^0(A/B)$ is a fine moduli space, one needs to rigidify the situation by assuming that
the restriction $\alpha|_{B}\in H^2(B,\Gm)$ of $\alpha$ to the zero section $B\subset A$ of $f\colon A\to B$ is trivial. 
 In this case, the standard Fourier--Mukai formalism adapted to the twisted case, cf.\ \cite{KLB},\TBC{Check reference.} with
 the $(1,\alpha)$-twisted Poincar\'e line bundle $$\kp_\alpha\to A_\alpha\times_B A,$$
 taken as a Fourier--Mukai kernel defines an equivalence
$$\Phi_{\kp_\alpha}\colon \Db(A_\alpha)\congpf  \Db(A,\alpha).$$

\begin{cor}\label{cor:piab}
Let $f\colon A\to B$ be an abelian scheme of relative dimension
$g$ and let $\alpha\in \Br_1(A)^0$ be a Brauer class with trivial restriction
to the zero section $B\subset A$. Furthermore, assume that $L$ is
a line bundle on $\check A$ of non-zero fibre degree, i.e.\ $\chi(\check A_t,L_t)\ne0$.

Then, if the period $r=\per(\alpha)$
is coprime to $\chi(\check A_t,L_{t})$ and the obstruction class $o(\alpha,L^r)\in H^2(B,\Gm)$ is trivial, one has
\begin{equation}\label{eqn:indper1}
\ind(\alpha)\mid\per(\alpha)^g.
\end{equation}
In particular, if $L$ is a principal polarisation, then {\rm (\ref{eqn:indper1})} holds for all $\alpha\in \Br_1(A)^0$ with vanishing obstruction.
\end{cor}

\begin{proof}
According to Proposition \ref{prop:lbabscheme}, there exists a line bundle $M_\alpha$ on $A_\alpha$ the fibre degree of which is 
$\chi(A_{\alpha t}, M_{\alpha t})=\per(\alpha)^g\cdot  \chi(\check A_t,L_t)$. Since $\Phi_\kp(M_\alpha)$ is a complex of $\alpha$-twisted sheaves on $A$ of total rank $\chi(A_{\alpha t}, {M_{\alpha t}})$, we can conclude $$\ind(\alpha)\mid\per(\alpha)^g\cdot  \chi(\check A_t,L_{t}).$$ Since $\ind(\alpha)$ and $\per(\alpha)$ have the same prime factors, this implies the assertion.
\end{proof}

Of course, this result is only useful if the vanishing of the obstruction class $o(\alpha,L^r)$ can
be effectively checked. For example, if the Brauer group $\Br(B)$ of $B$ is trivial, then
all obstructions are trivial. However, in the applications, $B$ is typically not projective
and compactified only by a codimension one  boundary. This makes the vanishing of $\Br(B)$
extremely unlikely as soon as $\dim(B)>1$.

\begin{remark}\label{rem:GeomAppl}
In the geometric application later, we start with a polarisation $\kl$ of the abelian scheme $A\to B$ and so we would like to rephrase the condition on $\per(\alpha)$ in terms of $\kl$ on $A$ and not in terms of $L$ on $\check A$. For this we briefly show how to construct a line bundle $L$ on $\check A$ such that $\chi(\check A_t,L_t)$ and $\chi(A_t,\kl_t)$ have the same prime factors. \smallskip
 
 There are in fact 
various ways to associate a line bundle  $L$ on $\check{A}$ with a given one $\kl$ on $A$. The quickest, which gives the smallest possible degree, is to define
$L\coloneqq \check{\kl~}$ as the dual of the determinant of the image of  $\kl$ under the Fourier--Mukai equivalence
$\Phi_\kp\colon\Db(A)\congpf \Db(\check A)$ given by the classical Poincar\'e line bundle $\kp$, i.e.$$L\coloneqq \det(\Phi_\kp(\kl))^\ast.$$ It is known that if $\kl$ is relatively ample, then also $L$ is. But even without assuming ampleness, one knows
$\chi(\check A_t, L_{t})= \chi(A_t,\kl_{t})^{g-1}$, which follows immediately from
a result of Mukai \cite[Prop.\ 3.11]{Mukai}, see also the discussion in \S\! \ref{sec:compactify} and \cite{Bi}. 

In the end we will not work with this specific $L$ but with a closely related
line bundle $M$ obtained from an ample line bundle on a modular compactification of $\check A$.
It will have the property that the three fibre degrees $\chi(\check A_t,M_t)$, $\chi(\check A_t,L_t)$, and $\chi(A_t,\kl_t)$ have the same prime factors which is all that matters.
\end{remark}

\begin{remark}\label{rem:onproof}
If $\kl$ and hence $L=\check{\kl~}$ are relatively ample, then also the line bundle $M_\alpha$ on $A_\alpha$
is relatively ample. In this case,
$\Phi_\kp(M_\alpha)$ is actually a locally free $\alpha$-twisted sheaf  of rank $\chi(A_{\alpha t},M_{\alpha t})$
and not merely a complex. \smallskip

Alternatively, as in \cite{HuyMa}, one could argue by taking global sections of the restriction of $\kp$ to $A_\alpha\times_B\eta_A$. This results in a twisted sheaf over the generic point $\eta\in A$
and $\ind(\alpha)=\ind(\alpha_\eta)$.
\end{remark}

\begin{remark}\label{rem:2g}
The weaker assertion $$\ind(\alpha)\mid\per(\alpha)^{2g}$$
always holds and is much easier to prove. Indeed, as has been mentioned in \S\! \ref{sec:Multi}, there always exists a cycle
$Z\subset A_\alpha$ of relative length $\per(\alpha)^{2g}$. The restriction $\kp_\alpha|_{ Z\times_BA}$
can the be viewed as an $\alpha$-twisted locally free sheaf on $A$, which implies the assertion.
\end{remark}

\begin{remark} The whole discussion so far could have been done in the analytic category. In this case,
the cohomological Brauer group $H^2(B,\ko^\ast)$ is typically much easier to control. For example,
by shrinking $B$ to an affine scheme, one can assume that $H^2(B,\ko)=0$, so that via the
exponential sequence $H^2(B,\ko^\ast)$ is contained in the finitely generated group $H^3(B,\ZZ)$.
Then, for $\per(\alpha)$ coprime to the order of the torsion of $H^3(B,\ZZ)$, the obstruction
$o(\alpha,L^r)$ is trivial. However, this would eventually only produce a line bundle $M_\alpha$ on $A_\alpha$
in the analytic category and via the analytic Fourier--Mukai equivalence
$\Db(A_\alpha)\congpf D^b(A,\alpha)$ an $\alpha$-twisted locally free sheaf on $A$ in the analytic category.
This would not be enough to draw any conclusion concerning the index in the \'etale setting.
\end{remark}
\subsection{Fibrations without a section}\label{sec:gone}
Let us indicate how to adapt the preceding discussion to the case of smooth fibrations
$f\colon A\to B$ with geometric fibres isomorphic to abelian varieties but without a global zero section.
The existence of a section of $A\to B$ was used twice to rigidify certain Picard schemes: 
\begin{enumerate}\setlength\itemsep{0.3em}
\item[(i)] For the existence of the classical Poincar\'e bundle $\kp$ on $\check A\times_B A$, which in turn was used to transform a line bundle $\kl$ on $A$ into a line bundle $L$ on $\check A$. 
\item[(ii)] For the existence of the twisted Poincar\'e bundle $\kp_\alpha$ on $A_\alpha\times_BA$, which also needed the restriction of $\alpha$ to the section
to be trivial. The latter was then used for the equivalence $\Phi_{\kp_\alpha}\colon \Db(A_\alpha)\congpf\Db(A,\alpha)$.
\end{enumerate}
\smallskip

In \S\! \ref{sec:HK}, it will become clear  that for (i)
the Poincar\'e bundle is not really needed. After all,  as a quasi-projective
variety, $\check A$ certainly admits an ample line bundle, we only need to control its degree. 
This approach will be pursued in \S\! \ref{sec:compactify}. There, the dual
abelian scheme $\check A$ will be viewed as an open subset of a moduli space of stable sheaves
and the line bundle naturally extends to a compactification. The extension will be crucial
to control the obstruction class.

However, (ii) is more serious, as in the application to hyperk\"ahler manifolds
one really has to face cases where $A\to B$ comes without a section. The solution will
be provided by Lagrangian multi-sections. Incidentally, 
the arguments addressing (ii) will also take care of (i), see Remark \ref{rem:LBnoSec} below.
\smallskip


Assuming projectivity of the morphism $f\colon A\to B$, there always exists a multi-section, i.e.\ an integral subscheme $\tilde B\subset A$ which is generically finite over $B$. Let us fix one and denote by $k$ its degree. We may think of $\tilde B$ as the closure of a certain closed point of degree $k$ in the  scheme-theoretic generic fibre $A_\eta$.
The base change
$$\tilde A\coloneqq A\times_B\tilde B\to \tilde B$$ is then an abelian scheme of the form considered before. Its zero section is given by the diagonal morphism $\tilde B\to A\times_B\tilde B$.
In particular, there exists a Poincar\'e bundle $\kp$ on  the base change $(A\times_B\check A)\times_B\tilde B$
inducing a Fourier--Mukai equivalence $\Phi_\kp\colon \Db(A\times_B\tilde B)\congpf\Db({\check A}\times_B\tilde B)$
relative over $B$.

Let us now add a Brauer class $\alpha\in \Br_1(A)^0$ and denote by $\tilde\alpha\in\Br(\tilde A)$ its pull-back under the projection 
$ \tilde A=A\times_B\tilde B\to A$, which  is then also contained
in the subgroup $\Br_1(\tilde A)^0$. In order to combine the above with the discussion in \S\! \ref{sec:ApplPI}, 
we need to assume that $\tilde\alpha$ is trivial along the zero section of $\tilde A=A\times_B\tilde B\to \tilde B$
or, equivalently, that the restriction $\alpha|_{\tilde B}$ to the multi-section $\tilde B\subset A$ is trivial.
In the application to Lagrangian fibrations of hyperk\"ahler varieties, we will work with one such $\tilde B$ that can be chosen independently of the 
Brauer class $\alpha$. 

\begin{cor}\label{cor:nosec}
Let $A\to B$ be of relative dimension $g$ with a multi-section $\tilde B\subset A$ of degree $k$.
Furthermore, let  $L$ be a line bundle on $\check A$ with non-zero fibre degree, i.e.\ 
$\chi(\check A_t,L_t)\ne0$.

Assume  $\alpha\in \Br_1(A)^0$ is a class with period $r=\per(\alpha)$  coprime to $
k\cdot\chi(\check A_t,L_t)$
and such that its restriction $\alpha|_{\tilde B}\in H^2(\tilde B,\Gm)$ and its obstruction class $o(\alpha,L^r)\in H^2(B,\Gm)$ are both trivial. Then
\begin{equation}\label{eqn:indper}
\ind(\alpha)\mid\per(\alpha)^g.
\end{equation}
\end{cor}

\begin{proof} Before we start, recall that the definition of $o(\alpha,L^r)\in H^2(B,\Gm)$ in \S\! \ref{sec:glueAS} did not assume the existence of a zero section of $A\to B$.
To simplify the discussion, we will assume that $L$ is relatively ample, but see the end of the proof
for comments on the general case.\smallskip

If $r=\per(\alpha)$ is coprime to $\chi(\check A_t,L_t)$, then also the period $\per(\tilde\alpha)$ of the
pull-back $\tilde\alpha$ is.
Hence, applying Corollary \ref{cor:piab} to $\tilde\alpha$ while
using (\ref{eqn:LLk}), one finds $\ind(\tilde\alpha)\mid\per(\tilde\alpha)^g$,
assuming that  $\per(\alpha)$ is coprime to $\chi(\check A_t,L_t)$ and that $o(\tilde\alpha,\pi^\ast L^r)$ is trivial.  More precisely, in this case there exists an $\tilde\alpha$-twisted locally free sheaf on $\tilde A$ of rank $\per(\alpha)^g$, cf.\ Remark \ref{rem:onproof}.

Taking its direct image under the projection $\tilde A\to A$, we obtain an $\alpha$-twisted sheaf on
$A$ of rank $k\cdot\per(\alpha)^g$. As $\tilde B\to B$ is only generically finite, this sheaf may only be locally free on some open subset of $A$ but this is enough to conclude $\ind(\alpha)\mid k\cdot\per(\alpha)^g$, which 
in turn proves the assertion.

To conclude, observe that the assumption $o(\alpha,L^r)\in H^2(B,\Gm)$ being trivial implies
that also its pull-back $o(\tilde\alpha,\pi^\ast L^r)\in H^2(\tilde B,\Gm)$ is trivial,
where $\pi\colon \check A\times_B\tilde B\to \check A$ is the projection.\smallskip

If $L$ is not relatively ample and just satisfies $\chi(\check A_t,L_t)\ne0$, then instead of an $\tilde\alpha$-twisted locally free sheaf on $\tilde A$ one obtains a complex of $\tilde\alpha$-twisted sheaves on $\tilde A$ and the direct image on $A$ will again be just a complex.
However, the conclusion is the same, for the index is defined as a gcd. 
\end{proof}

In the application to Lagrangian fibrations of hyperk\"ahler manifolds the line bundle $L$ will
be obtained as the restriction of an ample line bundle on a modular compactification of $\check A$.  However, building upon Remark \ref{rem:GeomAppl}, there is an alternative approach to produce
a convenient line bundle that avoids compactifying $A$ or $\check A$.
As this might be of independent interest, we shall record in the following remark.

\begin{remark}\label{rem:LBnoSec}
As above, we assume $\tilde B\subset A$ is a multi-section of $A\to B$ of degree $k$.
We claim that for any line bundle  $\kl$ on $A$,  there exists a line bundle $L$ on $\check A$ such that
$$\chi(\check A_t,L_{t})=k^g\cdot\chi(A_t,\kl_{t})^{g-1}.$$

Observe that with this line bundle $L$  the assumption in Corollary \ref{cor:nosec} that $r=\per(\alpha)$ be coprime to $k\cdot\chi(\check A_t,L_t)$
is  equivalent to $r$ being coprime to $k\cdot \chi(A_t,\kl_t)$.
The latter only involves the original fibration $A\to B$.
\smallskip

The proof of this assertion goes as follows: Let $\tilde{\kl~~}\!\!$ be the pull-back of $\kl$ under $\tilde A\to A$ and define
$\tilde L\coloneqq \det(\Phi_\kp(\tilde{\kl~~}\!\!))^\ast\in\Pic(\check A\times_B\tilde B)$.
Observe that for a closed point $\tilde t\in \tilde B$ mapping to $t\in B$, one has $(A\times_B\tilde B)_{\tilde t}\cong A_t$, $(\check A\times_B\tilde B)_{\tilde t}\cong \check A_t$, and $\chi(\check A_t, \tilde L_{t})= \chi(A_t,\kl_{t})^{g-1}$.

If  we now let $L\coloneqq\det(\pi_\ast\Phi_\kp(\tilde{\kl~~}\!\!))^\ast$, where $\pi\colon \check A\times_B\tilde B\to \check A$ is the projection, then
\begin{equation}\label{eqn:LLk}
L_{t}\cong \tilde L_{\tilde t}^k\text{ and }\chi(\check A_t, L_{t})= k^g\cdot \chi(A_t,\kl_{t})^{g-1},
\end{equation} which concludes the argument.

\end{remark}

\begin{remark}
Note that eventually the assumption that $\per(\alpha)$ and $k$ be coprime should be superfluous. Indeed,
for a genus one fibred surface $S\to C$ one knows $\ind(\alpha)=\per(\alpha)$ by \cite{dJ} without any further assumption on $\per(\alpha)$. The above techniques do not allow to draw this stronger conclusion.
One reason might be that equality here is shown by actually producing a locally free $\alpha$-twisted
sheaf of rank $\per(\alpha)$ which is a priori stronger than the assertion $\ind(\alpha)=\per(\alpha)$.
\end{remark}

\begin{remark}\label{rem:2gbis}
Following up on Remark \ref{rem:2g}, one again gets, and quite easily, the weaker
assertion $\ind(\alpha)\mid\per(\alpha)^{2g}$
for all $\alpha$ with $\per(\alpha)$ coprime to the degree $k$ of the $\alpha$-trivialising section
$\tilde B\subset A$. 

This observation applies directly to projective hyperk\"ahler manifolds $X$ that admit
a Lagrangian fibrations $X\to \PP$ and proves $\ind(\alpha)\mid\per(\alpha)^{\dim(X)}$
for all $\alpha\in \Br(X)$ with $\per(\alpha)$ coprime to the order of $H^3(X,\ZZ)_{\text{tors}}$ and
to the degree $\deg(\tilde \PP/\PP)$ of a Lagrangian multi-section $\tilde \PP\subset\PP$, see \S\! \ref{sec:rigidify}.
\end{remark}

\section{Period-index for Lagrangian fibrations}\label{sec:HK}
The result in the previous section can of course be applied to the smooth part $f\colon A\to B$ of any Lagrangian
fibration $$h\colon X\twoheadrightarrow\PP=\PP^n$$ of a projective hyperk\"ahler manifold $X$. In this case, the relative
dimension  of $A\to B$, denoted $g$ before, is just $n=\dim(X)/2$. Note that it is
conjectured and proved in many cases that the base of any Lagrangian fibration of a projective hyperk\"ahler manifold
is automatically smooth and then isomorphic to the projective space, cf.\ \cite{BS,HuyXu,Hwang,LiTo}. To simplify
the discussion we will just assume this from the outset. In fact, the smoothness of the base implies the flatness of the fibration and this is technically useful.
\smallskip

The structure of this section is as follows: In \S\!  \ref{sec:setup}
we compare the Brauer group of a hyperk\"ahler manifold with its analytic counterpart
and identify the finite index subgroup of those Brauer classes that can be treated with methods
of the previous section. In \S\! \ref{sec:compactify} we explain how to compactify the 
dual abelian scheme $\check A\to B$ to a projective moduli space
$\km\to \PP$ and how to extend at least numerically the natural ample line bundle
$L=\check{\kl\,\,}\!\!$ on $\check A$ used before to a line bundle $M$ on the compactification
$\km$. This slightly technical point is eventually needed to control the obstruction class in \S\! \ref{sec:VanishingOC}. There will be no need of compactifying the twists $A_\alpha$, which
in principle is possible, but we will need to have a universal family of twisted sheaves at our
disposal. This will be addressed in \S\! \ref{sec:rigidify}. 
In \S\! \ref{sec:proof} we combine everything to conclude the proof of Theorem \ref{thm1}.
In the last \S\! \ref{sec:CoveringCubic} we explain how the arguments can be adapted to 
cover Fano varieties of lines on cubic fourfolds, producing a uniform exponent
$5$ instead of the expected $2$.

\subsection{Setup for Theorem \ref{thm1}}\label{sec:setup}
Let $X$ be a complex projective manifold. Then using the usual comparison theorems
of \'etale and analytic cohomology with finite coefficients, the Brauer group $\Br(X)$ can
be seen to be naturally isomorphic to the torsion subgroup of the analytic Brauer group $H^2(X,\ko^\ast)$. The latter sits in the long exact exponential  sequence
$$\cdots\to H^2(X,\ZZ)\to H^2(X,\ko)\to H^2(X,\ko^\ast)\to H^3(X,\ZZ)\to\cdots,$$ where $H^2(X,\ko)$ is one-dimensional. 
The \emph{connected analytic Brauer group} is defined as the `one-dimensional' subgroup
$$ H^2(X,\ko^\ast)^0\coloneqq H^2(X,\ko)/\im(H^2(X,\ZZ)\to H^2(X,\ko))\subset H^2(X,\ko^\ast).$$
Furthermore, since the connected analytic Brauer group is divisible, the algebraic Brauer group $\Br(X)\subset H^2(X,\ko^\ast)$ maps onto $H^3(X,\ZZ)_{\text{tors}}$.
Thus, the intersection 
\begin{equation}\label{eqn:Br0}
\Br(X)^0\coloneqq\Br(X)\cap H^2(X,\ko^\ast)^0\subset\Br(X)
\end{equation}
is a subgroup of finite index
$|H^3(X,\ZZ)_{\text{tors}}|$. In particular, any $\alpha\in \Br(X)$ of order coprime to the order of
$H^3(X,\ZZ)_{\text{tors}}$ is necessarily contained in $\Br(X)^0$.

\begin{remark}\label{rem:TorsionHK}
In the literature one finds the following information about $H^3(X,\ZZ)_{\text{tors}}$ for
the four known deformation types of compact hyperk\"ahler manifolds:
\begin{enumerate}\setlength\itemsep{0.3em}
\item[(i)] The cohomology of the Hilbert scheme $S^{[n]}$ of a K3 surface is torsion free, cf.\ \cite{MarkInt,Tot}. In fact, $H^3(S^{[n]},\ZZ)=0$. 
\item[(ii)] The cohomology of the generalised Kummer variety $K_2(A)$ of dimension four is torsion free, cf.\ \cite{KaMe}. In fact, $H^3(K_2(A),\ZZ)\cong \ZZ^{\oplus 4}$. This is expected to hold for all $K_n(A)$.\footnote{Torsion freeness was originally claimed in all dimensions, but a problem with the proof was pointed out by Totaro, see the Corrigendum.}
\item[(iii)] The odd Betti numbers of varieties of OG6-type are trivial, cf.\ \cite{MoRaSa}. It is expected that $H^3(X,\ZZ)=0$.
\item[(iv)] The odd Betti numbers of varieties of OG10-type are trivial, cf.\ \cite{DCaRaSa}. It is expected that $H^3(X,\ZZ)=0$.
\end{enumerate}
So it seems, that at the moment no compact hyperk\"ahler manifold is known that would admit non-trivial torsion classes in $H^3(X,\ZZ)$. However,  there is no a priori reason why  $\Br(X)^0=\Br(X)$ should always hold.
\end{remark}

By definition of a Lagrangian fibration, the restriction maps $H^0(X,\Omega_X^2)\to H^0(A_t,\Omega^2_{A_t})$ and their complex conjugates
$$H^2(X,\ko)\to H^2(A,\ko)\to H^2(A_t,\ko)$$  to the smooth fibres $A_t$
are trivial. In particular, the image of
$\Br(X)^0\to \Br(A)$ is contained in $\Br_1(A)$. Altogether this proves that restriction
to the open subset $A\subset X$ defines a map
$$\Br(X)^0\to \Br_1(A)^0,$$ so that Corollary \ref{cor:piab} applies to all classes in $\Br(X)^0$ and thus to
all classes of order coprime to $|H^3(X,\ZZ)_{\text{tors}}|$. 
Since in the Leray spectral sequence $H^2(X,\ko_X^\ast)\to H^1(B,R^1f_\ast(\ko_X^\ast))$ maps
$H^2(X,\ko_X^\ast)^0$ into $H^1(B, (R^1f_\ast(\ko_X^\ast))^0)$ (use continuity), one obtains the map $$\Br(X)^0\to H^1(B,(R^1f_\ast\Gm)^0),$$ so that the discussion of the preceding sections applies.

\subsection{Line bundles on the compactified dual abelian scheme}\label{sec:compactify}
We continue to consider a Lagrangian fibration $h\colon X\to\PP$ of a projective hyperk\"ahler manifold $X$. We  fix in addition an ample line bundle $\kl$ on $X$ and also use
it as a polarisation of the smooth part $f\colon A\to B$, which may not have a section.\smallskip

The dual abelian scheme $\check A\to B$
 parametrises numerically trivial line bundles on the fibres
$A_t$. Since these fibres are integral, any line bundle on one of them, considered as a sheaf on the ambient projective
variety $X$, is stable with respect to any polarisation. Hence, $\check A$ can be viewed as an open subset
of the locally projective moduli space of all sheaves that are semi-stable with respect to the polarisation $\kl$, see \cite{HuLe,Simpson} for the construction and further references.
Let us denote by $\km$ the irreducible component of this moduli space that contains $\check A$. So, $\km$ is now a projective variety for which the support map  $\km\to \PP$ extends the twisted Picard scheme:
$$\xymatrix@R=1pt{\check A\ar[r]&B\\
\cap&\cap\\
\km\ar[r]&\PP.}$$ 

The existence of the support map requires the morphism $h\colon X\to\PP$ to be flat which
is in fact equivalent to the smoothness of the base, cf.\ \cite[Rem.\ 1.18]{HuyMau}.
\smallskip

%

Of course, the projective moduli space $\km$ comes with a (non-unique) ample line
bundle. To prove Theorem \ref{thm1} with an unspecified integer $N_X$,  this is all that is needed. However, if one wants to be more specific, e.g.\ as in Corollaries \ref{cor:piab} and \ref{cor:nosec}, then we need control over the degree of the ample line bundles
on the fibres $\check A_t\subset \check A\subset \km$. This should be considered
as a technical issue and the reader is encouraged to skip the proof of the next result
at first reading. Its proof will occupy the rest of the section.

\begin{prop}\label{prop:lbM}
There exists a line bundle $M$ on $\km$ such that its restriction $M_t\coloneqq M_{\check A_t}$
satisfies $\chi(\check A_t,M_t)=m_1\cdot\chi(A_t,\kl_t)^{m_2}$ for certain integers
$m_1$ and $m_2$. Moreover, the coefficient $m_1$ can be chosen coprime
to any given integer $r$.
\end{prop}

The basic idea is to compare the polarisations of $\km$ given by its description as a moduli space with the polarisation of the dual abelian
scheme $\check A\to B$ discussed before. We start by recalling both settings.

\smallskip

$\bullet$ Let us first consider again the dual abelian scheme $\check A\to B$. If $A\to B$ admits
a section, then a universal Poincar\'e bundle $\kp$ on $\check A\times_BA$ exists and
it is unique if appropriately normalised. However, to describe it explicitly on a fibre $A_t$,
it is common to use a polarisation $\kl$ on $A$, its restriction $\kl_t\coloneqq\kl_{A_t}$ on $A_t$,
and the induced morphism $\varphi_{\kl_{t}}\colon A_t\to\check A_t$.
Then, cf.\ \cite[\S\! 13]{MumAV}: 
\begin{equation}\label{eqn:Poinc}
(\varphi_{\kl_{t}}\times {\rm id})^\ast\kp|_{\check A_t\times A_t}\cong m^\ast\kl_{{t}}\otimes p_1^\ast\kl_{{t}}^{-1}\otimes
p_2^\ast\kl_{{t}}^{-1}.
\end{equation}
Here, $m\colon A_t\times A_t\to A_t$ is the multiplication and $p_1,p_2$
denote the two projections. Used as a Fourier--Mukai kernel, the Poincar\'e bundle $\kp$
induces an equivalence $\Phi_\kp\colon \Db(A)\congpf\Db(\check A)$, which
we have used before to define a polarisation $L\coloneqq\check{\kl\,\,}\!\!$ on $\check A$
as the dual of the determinant of $\Phi_\kp(\kl)$. To control the degree  of the restriction
$L_t\coloneqq L|_{\check A_t}$, we express  the
pull-back $\varphi_{\kl_t}^\ast L_t$ using a result of Mukai \cite[Prop.\ 3.11]{Mukai}.
 Fibrewise it says
\begin{equation}\label{eqn:Mukai}
\varphi_{\kl_t}^\ast \Phi_\kp(\kl_t)\cong H^0(A_t,\kl_{t})\otimes \kl^\ast_{t}\text{ and hence }
\varphi_{\kl_t}^\ast L_t\cong \kl_t^{\chi(\kl_t)}
\end{equation}
which together with $\deg(\varphi_{\kl_t})=\chi(A_t,\kl_t)^2$ implies $\chi(\check A_t,L_{t})=\chi(A_t,\kl_{t})^{g-1}$, see also \cite{Bi}. 

Note that in the following, we will not use the existence
of a zero section of $A\to B$ nor the existence of $\kp$ on $\check A\times_BA$. All computations will be done on the geometric fibres $A_t$ and $\check A_t$ where the
Poincar\'e bundle $\kp_t$ on $\check A_t\times A_t$ always exists.
\smallskip

$\bullet$ Alternatively, we can treat $\check A$ as (an open subset of) a moduli space of stable sheaves on $A$ or $X$ and as such it is constructed as a GIT quotient of some Quot-scheme, see \cite{HuLe,Simpson} for details and references concerning the construction of moduli spaces of
semi-stable sheaves. More
precisely, we consider all sheaves parametrised by $\check A$, so numerically trivial line bundles $P$
on the fibres $A_t$, as quotients $(\kl_{t}^{-N})^{\oplus \chi_N}\twoo P$, where $\chi_N\coloneqq \chi(\kl_{t}^N\otimes P)=\chi(\kl_{t}^N)=N^g\cdot\chi(\kl_t)$ and $N$ is chosen sufficiently big. At this point in the discussion, $N> 2$ is sufficient, but we might need higher powers later for the compactification $\km$. 

Then there is an open subset ${\mathcal R}_B\subset \text{Quot}_B$ of the relative Quot-scheme 
$\text{Quot}_B\to B$ of quotients 
of $(\kl_{t}^{-N})^{\oplus \chi_N}$ parametrising only line bundles on the fibres $A_t$ such that
$\check A$ is a good quotient by the natural $\text{PGL}(\chi_N)$-action
\begin{equation}\label{eqn:GIT}
\pi\colon{\mathcal R}_B\twoo \check A={\mathcal R}_B/\!/\,\text{PGL}(\chi_N).
\end{equation}
 
 The Quot-scheme representing the Quot-functor  comes with a universal quotient sheaf
 $\tilde \kp$ on $\text{Quot}_B\times_BA$ which is used to identify relatively ample line bundles
 \begin{equation}\label{eqn:ampleMl}
 \tilde M_\ell\coloneqq\det p_{1\ast}(\tilde\kp\otimes p_2^\ast \kl^\ell),
 \end{equation}$\ell\gg0$, on $\text{Quot}_B$, see \cite[Ch.\ 4.3]{HuLe}. The GIT quotient
 (\ref{eqn:GIT}) is in fact formed with respect to the linearisation given by $\tilde M_\ell$ and, therefore,  $\tilde M_\ell$ descends to an ample line bundles $M_\ell$ on $\check A$.   \smallskip
 
 While $L=\check{\kl\,\,}\!\!$ on $\check A$ is easier to construct, it is
 the line bundles $M_\ell$ that naturally extend to line bundles on $\km$ under the inclusion $\check A\subset \km$ and we will need to have control over their degree on the fibres. This is achieved by comparing the line bundles $M_\ell$ to $L$ and the first step is comparing 
 the restrictions  $\tilde \kp_t\coloneqq \tilde\kp|_{\text{Quot}_t\times A_t}$ and $\kp_t\coloneqq \kp|_{\check A_t\times A_t}$.
 \smallskip

Choose a basis of $H^0(A_t,\kl_t^N)$ and compose it with the evaluation map
to produce a surjection  $\ko_{A_t}^{\oplus \chi_N}\cong H^0(A_t,\kl_t^N)\otimes\ko_{A_t}\twoo\kl_t^N$. Then pulling back under the multiplication
$m\colon A_t\times A_t\to A_t$, tensoring with $p_2^\ast\kl_t^{-N}$,  and applying (\ref{eqn:Poinc}) yields a surjection 
$$(\ko\boxtimes\kl_t^{-N})^{\oplus\chi_N}\twoo (\varphi_{\kl_t^N}\times{\rm id})^\ast(\kp_t)\otimes p_1^\ast\kl_t^{N}$$
on $A_t\times A_t$, which in turn defines a classifying morphism
$\tilde\psi\colon A_t\to \text{Quot}_t$ with $$(\tilde\psi\times{\rm id})^\ast\tilde\kp_t\cong
(\varphi_{\kl_t^N}\times{\rm id})^\ast(\kp_t)\otimes p_1^\ast\kl_t^{N}$$
on $A_t\times A_t$. 
Note that by construction, the composition $\psi\coloneqq\pi_t\circ\psi\colon A_t\to \text{Quot}_t\to\check A_t$ is simply $\varphi_{\kl^N_t}$.\smallskip

With this we can now numerically identify the ample line bundle $M_\ell$ on $\check A$  
on the fibres.

\begin{lem} 
For $\ell=N\cdot\ell'\gg0$, the line bundle  $M_{\ell t}$ on $\check A_t$ is ample
with $$\chi(\check A_t,M_{\ell t})=N^{g^2-g}\cdot\ell'^{g^2-g}\cdot(\ell'-1)^g\cdot \chi(A_t,\kl_t)^{g-1}.$$
\end{lem}

For the application we have in mind the precise form of the right hand side
 is of no importance. The only thing that
matters is that first $N\gg0$ and then $\ell'\gg0$ can be chosen such that any $r$ coprime
to $\chi(A_t,\kl_t)$ will also be coprime to $\chi(\check A_t,M_{\ell t})$.

\begin{proof}
By the above and using the shorthand $\chi_\ell\coloneqq \chi(A_t,\kl_t^\ell)$, one has
\begin{eqnarray*}\varphi_{\kl^N_t}^\ast(M_{\ell t})&\cong& \psi^\ast( M_{\ell t})\cong\tilde\psi^\ast (\tilde M_{\ell t})
\cong\det\tilde\psi^\ast(p_{1\ast}(\tilde\kp_t\otimes p_2^\ast\kl_t^\ell))\\
&\cong&
\det p_{1\ast}\left((\tilde\psi\times{\rm id})^\ast\tilde\kp_t\otimes p_2^\ast\kl_t^\ell\right)\\
&\cong& \det p_{1\ast}\left((\varphi_{\kl_t^N}\times{\rm id})^\ast\kp_t\otimes p_1^\ast\kl_t^N\otimes p_2^\ast\kl_t^\ell\right)\\&\cong&\det \left( p_{1\ast}\left((\varphi_{\kl_t^N}\times{\rm id})^\ast\kp_t\otimes p_2^\ast\kl_t^\ell\right)\otimes \kl_t^N\right)\\
&\cong&\det\left(\varphi_{\kl_t^N}^\ast\Phi_{\kp_t}(\kl_t^\ell)\otimes \kl_t^N\right)\\
&\cong&\varphi_{\kl_t^N}^\ast\det\left(\Phi_{\kp_t}(\kl_t^\ell)\right)\otimes\kl_t^{N\cdot\chi_\ell}.
\end{eqnarray*}
For $\ell=N\cdot \ell'$ we write $\varphi_{\kl_t^\ell}\colon A_t\to \check A_t$ as the composition $(\ell'\cdot ~) \circ\varphi_{\kl_t^N}\colon A_t\to\check A_t\to \check A_t$
and use (\ref{eqn:Mukai}) to compute
\begin{eqnarray*}\kl_t^{-\ell\cdot\chi_\ell}&\cong&\det((\kl_t^{-\ell})^{\oplus \chi_\ell})\cong\det\varphi_{\kl_t^\ell}^\ast\Phi_{\kp_t}(\kl_t^\ell)\\&\cong&\det\varphi^\ast_{\kl_t^N}(\ell'\cdot)^\ast\Phi_{\kp_t}(\kl_t^\ell)\cong\varphi^\ast_{\kl_t^N}((\ell'\cdot)^\ast\det\Phi_{\kp_t}(\kl_t^\ell))\\
&\cong&\varphi^\ast_{\kl_t^N}\left((\det\Phi_{\kp_t}(\kl_t^\ell))^{\ell'(\ell'+1)/2}\otimes (-1)^\ast(\det\Phi_{\kp_t}(\kl_t^\ell))^{\ell'(\ell'-1)/2})\right),
\end{eqnarray*}
where we use the standard formula for $(n\cdot)^\ast L$ of a line bundle $L$ on an abelian variety.
One could speculate that $\det \Phi_{\kp_t}(\kl_t^\ell)$ is in fact symmetric, in which
case one would have
\begin{equation}\label{eqn:aftersym}
\kl_t^{-\ell\cdot\chi_\ell}\cong\varphi^\ast_{\kl_t^N}\left(\det\Phi_{\kp_t}(\kl_t^\ell)\right)^{\ell'^2},
\end{equation}
but in any case we know that (\ref{eqn:aftersym}) is always true up to numerical equivalence. This then shows that $\varphi_{\kl_t^N}^\ast(M_{\ell t})^{\ell'^2}$ is numerically equivalent to $\kl_t^{\ell^{g+1}\cdot\chi_1\cdot(\ell'-1)}$, which is enough to conclude
\begin{eqnarray*}
\chi(\check A_t,M_{\ell t})&=&(\deg(\varphi_{\kl_t^N})\cdot \ell'^{2g})^{-1}\cdot\chi(A_t,\varphi_{\kl_t^N}^\ast(M_{\ell t})^{\ell'^2})\\
&=&(\deg(\varphi_{\kl_t^N})\cdot \ell'^{2g})^{-1}
\cdot(\ell^{g+1}\cdot\chi_1\cdot(\ell'-1))^g\cdot \chi(A_t, \kl_t)\\
&=&\ell^{g^2-g}\cdot(\ell'-1)^g\cdot\chi(A_t,\kl_t)^{g-1}
\end{eqnarray*}
and thus proving the assertion of the lemma.
\end{proof}

We can now complete the proof of Proposition \ref{prop:lbM}. The moduli space $\km$
is constructed as a GIT quotient of some open subset of semi-stable points
${\mathcal R}\subset \text{Quot}$
$${\mathcal R}\twoo \km={\mathcal R}/\!/\,\text{PGL}(\chi_N)$$ extending (\ref{eqn:GIT}) that described $\check A$ as a GIT-quotient. The universal quotient $\tilde\kp$ on $\text{Quot}_B\times_BA$ extends to a universal quotient
on $\text{Quot}\times_\PP X$ and the formula (\ref{eqn:ampleMl}) also describes
ample line bundles on $\km$, again denoted $M_\ell$. We now apply the
previous lemma with $m_1=\ell^{g^2-g}\cdot(\ell'-1)^g=N^{g^2-g}\cdot\ell'^{g^2-g}\cdot(\ell'-1)^g$ and $m_2=g-1$ combined with the observation that this $m_1$ can obviously assumed to be coprime to any given $r$ with $N$ and then $\ell'$ are chosen sufficiently large.\qed

\EV{
\begin{remark}
As discussed in \S\! \ref{sec:gone} already, the moduli space $\check A$ need not be fine and hence also $M$ has no reason to be fine.
 There we argued that $\check A$ can be turned into a fine moduli space after a generically finite base change $\tilde B\to B$
 to a multi-section of $f\colon A\to B$.
The closure $\tilde\PP$ of $\tilde B$ in $X$ can be viewed as a generically finite multi-section $\tilde \PP\subset X$
of the projection $h\colon X\to \PP$. In fact, as for the Picard scheme, the existence of the multi-section 
$\tilde \PP$ immediately implies that the base change $M\times_\PP\tilde \PP$ is a fine moduli space parametrising
only stable sheaves supported on the fibres of $X\times_\PP\tilde\PP\to\tilde \PP$. In other words,  we may assume that 
the Poincar\'e bundle on $(\check A\times_B A)\times_B\tilde B$ extends to a universal
sheaf $\ke$ on $(M\times_\PP X)\times_\PP\tilde \PP$. The Poincar\'e bundle and the universal sheaf can both be used
to produce line bundles on $\check A\times_B\tilde B$ and $M\times_\PP\tilde \PP$ out of a given line bundle $\kl$
on $X$ and its restriction to $A$, cf.\ the discussion in \S\! \ref{sec:gone}.

Explicitly, if $\kl$ is a relative ample line bundle on $X$ and the multi-section $\tilde B\to B$ is of degree $k$,
then there exists a line bundle $L$ on $M$ such that their degrees on the generic closed
fibre compare to each other via $\chi(M_t=\check A_t,L_{\check A_t})=k^g\cdot\chi(X_t=A_t,\kl_{A_t})^{g-1}$.
\end{remark}}

\subsection{Vanishing of the obstruction class}\label{sec:GLUE}\label{sec:VanishingOC}
We are now taking the next step towards extending the obstruction $o(\alpha,L^r)\in\Br(B)$ from $B$ to $\PP$ with the ultimate goal to show its triviality by using the triviality of $\Br(\PP)$.
Here, $L$ denotes the restriction of the ample line bundle $M$ on the compactification $\check A\subset\km$ as in Proposition \ref{prop:lbM}, so depending on the period $r=\per(\alpha)$.
There are again similarities between the discussion in this section  and Abasheva's results \cite[\S\! 3.0.3]{Abash}. The main difference is that Abasheva twists an actual hyperk\"ahler manifold while we twist a moduli space $\km$ of sheaves on a hyperk\"ahler manifold
with $\km$ usually not even be smooth.\smallskip

Recall that the obstruction class $o(\alpha,L^r)$ was obtained as the image of a certain class $\alpha'$ under
the boundary map $\delta\colon H^1(B, \underline{\check A}[r])\to H^2(B,\GG_m)$ induced
by the short exact sequence  (\ref{eqn:sesAV}),
which, using the identification $\underline{\check{A}}[r]=R^1f_\ast\mu_r$,
can alternatively be written as 
\begin{equation}\label{eqn:ses1}
\xymatrix{0\ar[r]&\Gm\ar[r]&\underline{\Aut}((\check{A},L^r)/B)[r]\ar[r]&R^1f_\ast\mu_r\ar[r]&0\,.}
\end{equation}

We then proceed in two steps:
\begin{enumerate}\setlength\itemsep{0.3em}
\item[(i)] The short exact sequence (\ref{eqn:ses1}) extends from $B$ to a
short exact sequence on $\PP$:
\begin{equation}\label{eqn:ses}
\xymatrix{0\ar[r]&\Gm\ar[r]&\ast\ar[r]&R^1h_\ast\mu_r\ar[r]&0\,.}
\end{equation}
In fact, for technical reasons the quotient sheaf $R^1h_\ast\mu_r$ will be replaced by a
certain subsheaf $\kk_r\subset R^1h_\ast\mu_r$ which still restricts to
$R^1f_\ast\mu_r$.
\item[(ii)] The class $\alpha'\in H^1(B,R^1f_\ast\mu_r)$ is the restriction of a class in $H^1(\PP,R^1h_\ast\mu_r)$ (or rather in $H^1(\PP,\kk_r)$).
\end{enumerate}

Let us begin with (i) for which we will make use of a slightly bigger open subset 
$$\check A\subset P\subset \km,$$
parametrising only stable sheaves on $X$ that are line bundles on the fibres $X_t$
which are in addition numerically trivial on each component.
Note that then $P\to \PP$ is a group scheme and clearly surjective.
Moreover, tensor product on the fibres defines an action $P\times_\PP \km\to \km$
which we use to prove the following result.

\begin{lem}
There exists a natural map 
\begin{equation}\label{eqn:map}
R^1h_\ast\mu_r\to \underline\Aut(\km/\PP)
\end{equation}
extending $R^1f_\ast\mu_r\hookrightarrow (R^1f_\ast\Gm)^0\cong
 \underline{\Aut}^0(\check{A}/B)$,  cf.\ {\rm(\ref{eqn:WC})}.
\end{lem}

\begin{proof} The sheaf $R^1h_\ast\mu_r$ is the sheafification of the pre-sheaf $U\mapsto H^1(X_U,\mu_r)$.
Since the image of the natural map $H^1(X_U,\mu_r)\to H^1(X_U,\GG_m)\cong\Pic(X_U)$ parametrises only torsion line
bundles, it is contained in the group scheme $P_U\to U$. The 
 action of the group scheme $P$ on $\km$, then induces an action of $H^1(X_U,\mu_r)$ on $\km_U$ and, therefore, a map (\ref{eqn:map}).
\end{proof}

In the analytic topology, we can view $\mu_r$ as a quotient
$0\to\ZZ\to \frac{1}{r}\ZZ\to \mu_r\to 0$ which allows us to define
a subsheaf  $\kk_r\subset R^1h_\ast\mu_r$ as the kernel of the boundary map,
in other words $$\kk_r\coloneqq\ker\left(R^1h_\ast\mu_r\to R^2h_\ast\ZZ\right)\cong \coker\left(R^1h_\ast\ZZ\to R^1h_\ast((1/r)\ZZ)\right).$$
Composing the inclusion with (\ref{eqn:map}) produces
\begin{equation}\label{eqn:map2}
\kk_r\to \underline\Aut(\km/\PP).
\end{equation}

Let us now fix a positive integer $r$ coprime to $\chi(A_t,\kl_t)$, which will later be the
period of the Brauer class, and let $M$ be an ample line bundle on the 
moduli space $\km$ as in Proposition \ref{prop:lbM} with $m_1$ coprime to $r$.
Then we consider the sheaf of automorphisms of $\km\to \PP$ together with the polarisation $M^r$ and its natural projection
\begin{equation}\label{eqn:mapp}
\pi\colon\underline\Aut((\km,M^r)/\PP)\to \underline\Aut(\km/\PP).
\end{equation}
As $\PP$ is normal, the morphism $\km\to \PP$ is connected and,
hence, the kernel of (\ref{eqn:mapp}) is $\Gm$. Thus,
 we have a short exact sequence
 \begin{equation}\label{eqn:mapses}
 \xymatrix{0\ar[r]&\Gm\ar[r]& \underline\Aut((\km,M^r)/\PP)\ar[r]&\im(\pi)\ar[r]&0.}
 \end{equation}
 
 \begin{lem}
The image of {\rm (\ref{eqn:map2})} is contained in the image of {\rm (\ref{eqn:mapp})}. 
\end{lem}

\begin{proof} Ideally one would like to prove that already the image of (\ref{eqn:map}) is contained in the image of (\ref{eqn:mapp}), but a priori there is an obstruction to this 
and we will show that it vanishes on the subsheaf $\kk_r$.\smallskip

Using the arguments and the notation in the proof of the previous lemma, it would be enough to show that tensor product with an $r$-torsion line bundle $K\in H^1(X_U,\Gm)$ preserves $L^r$. 
For this, observe that the pull-back $(K\otimes(~~))^\ast L^r$ of $L^r$ under $K\otimes(~~) \colon \km_U\to\km_U$ and the original line bundle $L^r$ are isomorphic  over the pre-image of $U\cap B\subset U$, see the proof of Proposition \ref{prop:lbabscheme}.

In order to conclude that they
are isomorphic on all of $\km_U$ it suffices to argue that they are numerically equivalent also
on the fibres $\km_t$ for $t\in U\,\setminus\, B$. Now, if the torsion line bundle $K\in \Pic(X_U)$ could be continuously deformed (through analytic line bundles) to the trivial line bundle, this would be immediate. By
means of the exponential sequence, we know that the obstruction to do this in the analytic category is the first Chern class ${\rm c}_1(K)\in H^2(X_U,\ZZ)$, but by definition  this
is guaranteed for all local sections of the subsheaf $\kk_r\subset R^1h_\ast\mu_r$.
\end{proof}

This allows us to pull-back (\ref{eqn:mapses}) via (\ref{eqn:map}) to obtain a short exact
sequence which we write as
$$ \xymatrix{0\ar[r]&\Gm\ar[r]& \underline\Aut((\km,M^r)/\PP)[r]\ar[r]&\kk_r\ar[r]&0.}$$
The restriction to $B$ gives back (\ref{eqn:ses1}) and this concludes the discussion of (i).\smallskip


Let us  now turn to (ii), so we want to show that the class $\alpha'\in H^1(B,R^1f_\ast\mu_r)$
is in the image of the restriction map
$$H^1(\PP,\kk_r)\to H^1(B,\kk_r|_B\cong R^1f_\ast\mu_r).$$
Clearly, then the obstruction class $o(\alpha,L^r)\in H^2(B,\Gm)$ is in the image of the restriction map $0=H^2(\PP,\Gm)\to H^2(B,\Gm)$ and hence trivial.
\smallskip

Recall that $\alpha'\in H^1(B,R^1f_\ast\mu_r)$ was any pre-image of $\alpha\in H^1(B,\underline{\check A})$ under the natural map induced by the inclusion $R^1f_\ast\mu_r\cong\underline{\check A}[r]\subset\underline{\check A}$. We shall now
explain that in the hyperk\"ahler setting there is a natural choice for $\alpha'$ that comes
from a global class on $X$ which then implies that $\alpha'$ is in fact the restriction of a class
in $H^1(\PP,\kk_r)$.\smallskip

For this we go back to the exponential sequence in \S\! \ref{sec:setup}. It
shows that the $r$-torsion part of the connected component of the Brauer group
can be identified as
$$\Br(X)^0[r]\cong \left((1/r)T(X)'\right)/\,T(X)'.$$

Here, $T(X)'$ is the cokernel of the inclusion ${\rm NS}(X)\,\hookrightarrow H^2(X,\ZZ)$
which contains the transcendental lattice $T(X)={\rm NS}(X)^\perp\subset H^2(X,\ZZ)$ as a
finite index subgroup. As the transcendental lattice is an irreducible
Hodge structure and the smooth fibres $X_t\subset X$ are Lagrangians,
the restriction $T(X)\to H^2(X_t,\ZZ)$ is trivial. Hence, the image of $T(X)\to H^0(\PP,R^2h_\ast \ZZ)$ contains only sections with support in the discriminant locus $\PP\,\setminus\, B\subset \PP$.
Presumably, the image is actually trivial but in any case it is torsion. Indeed, 
the kernel $T(X)_0\coloneqq\ker(T(X)\to H^0(\PP,R^2h_\ast\ZZ))$ is a non-trivial
sub-Hodge structure, cf.\ \cite[Prop.\ 4.5]{AbRo}, and since $T(X)\otimes\QQ$ is an irreducible
Hodge structure and so $T(X)_0\otimes\QQ=T(X)\otimes\QQ$, the inclusion $T(X)_0\subset T(X)$ has finite index. Altogether, this defines a sub-lattices $T(X)_0\subset T(X)'$ of  finite index 
and if $r$ is coprime to this index, then 
$$\Br(X)^0[r]\cong \left((1/r)T(X)'\right)/\,T(X)'\cong\left((1/r)T(X)_0\right)/\,T(X)_0.$$

We continue to work with $T(X)_0$ and use  the Leray spectral sequence to define the  maps
$$T(X)_0\to H^1(\PP, R^1h_\ast\ZZ) \text{ and } (1/r)T(X)_0\to H^1(\PP, (1/r) R^1h_\ast\ZZ), $$
which immediately yield the map
$$\left((1/r)T(X)_0\right)/\,T(X)_0\to H^1(\PP,\kk_r).$$
All that is left to do is to show the following.

\begin{lem}\label{lem:liftT0}
If $\alpha\in\Br(X)^0[r]$  lifts to $\beta\in (1/r)T(X)_0$, then the lift
$\alpha'\in H^1(B,R^1f_\ast\mu_r)$ of $\alpha\in H^1(B,\underline{\check{A}})$ can be chosen to be the image of $\beta$
under $(1/r)T(X)_0\to H^1(\PP,\kk_r)\to H^1(B,R^1h_\ast\mu_r)$.
\end{lem}

\begin{proof}
This is a consequence of the commutativity of the following diagram:
$$\xymatrix@R=18pt@C=15pt{(1/r)T(X)_0\ar[d]\,\ar[r]&\ker\ar[d]\ar[dr](H^2(X,\mu_r)\ar[r]& H^0(\PP,R^2h_\ast\mu_r))\\
H^1(\PP,\kk_r)\ar[r]&H^1(\PP,R^1f_\ast\mu_r)\ar[d]& H^2(X,\Gm)^0\ar[d]\\
&H^1(B,R^1f_\ast\mu_r)\ar[r]&H^1(B,R^1f_\ast\Gm)}$$
which follows from the functoriality of the Leray spectral sequence.
\end{proof}
\subsection{Rigidifying the twists}\label{sec:rigidify}
The last step in setting up the moduli spaces for our purposes is to rigidify the twisted Picard varieties $A_\alpha$.\smallskip

We use once again that Brauer classes on a hyperk\"ahler manifold $X$ automatically vanish on  Lagrangian subvarieties. We us a result of Lin \cite{Lin},\footnote{Thanks to C.\ Voisin for reminding me of Lin's result.} according to which there exists a constant cycle multi-section
$\tilde \PP\subset X$ of the projection $X\to \PP^n$, i.e.\ all closed points of $\tilde \PP$ define the same class in $\CH_0(X)$. Clearly, every actual section is automatically Lagrangian (and constant cycle). To avoid working with the a priori singular $\tilde\PP$, we
pass to a desingularisation. Now apply the Bloch--Srinivas principle, cf.\ \cite[Prop.\ 22.24]{VoisinHodge}, which for a constant
cycle subvariety like $\tilde \PP$ shows that the restriction map $H^2(X,\ko)\to H^2(\tilde \PP,\ko)$ is automatically trivial, which can also be expressed by saying that $\tilde \PP\to X$ is  Lagrangian.  Hence, $\Br(X)^0\to \Br(\tilde \PP)$ is trivial. \smallskip

This global information ensures that for the multi-section $\tilde B\coloneqq \tilde\PP\cap A$ 
of $A\to B$ the composition $\Br(X)^0\to\Br(A)\to \Br(\tilde B)$ is also trivial. Thus, $\tilde B$ can be used to rigidify all twisted
Picard varieties $A_\alpha\to B$ for $\alpha\in \Br(A)$ obtained by restricting classes in $\Br(X)^0$. Let us spell this out in more detail:
If we denote by $\tilde A$ and $\tilde A_\alpha$ the base changes $ A\times_B\tilde B$ resp.\ 
$A_\alpha\times_B\tilde B$, then there exists a Poincar\'e bundle: $$\kp\to A_\alpha\times_B\tilde A\cong \tilde A_\alpha\times_{\tilde B}\tilde A.$$
By definition, $\kp$ is twisted with respect to the pull-back $(1,\alpha)$ of $\alpha$ via the
composition of the second projection with $\tilde A\to A$. Note that $\tilde A_\alpha$ is nothing
but the twist of the dual of $\tilde A\to \tilde B$ with respect to the pull-back $\tilde\alpha\in \Br(\tilde A)$ of $\alpha$,
i.e.\ $\tilde A_\alpha=\Pic_{\tilde\alpha}^0(\tilde A/\tilde B)$.\smallskip

Then, any object in $\Db(\tilde A_\alpha)$ can be used via the Fourier--Mukai equivalence
$$\Db(\tilde A_\alpha)\cong \Db(\tilde A,\tilde \alpha)$$ to produce
an $\tilde\alpha$-twisted sheaf (or rather a complex of twisted sheaves) on $\tilde A$. The direct image under
$\tilde A\to A$ then turns such an $\tilde\alpha$-twisted sheaf on $\tilde A$ into an $\alpha$-twisted sheaf on $A$. In the process, the rank gets multiplied by the degree of $\tilde A\to A$, which is of course just the degree of the multi-sections $\tilde B\to B$ or, equivalently, of $\tilde \PP\to\PP$, and independent of the particular $\alpha$.

\subsection{Conclusion}\label{sec:proof}
We now reassemble the pieces to conclude the proof of Theorem \ref{thm1}.\smallskip

Assume $\alpha\in \Br(X)$ has a period $r\coloneqq\per(\alpha)$ which is coprime to the order of the torsion of $H^3(X,\ZZ)$, then $\alpha\in \Br(X)^0$, cf.\ (\ref{eqn:Br0}) in \S\! \ref{sec:setup}.  If, in addition, $\per(\alpha)$ 
is coprime to the index of $T(X)_0\subset H^2(X,\ZZ)/\NS(X)$, then $\alpha$ is induced by
a class $\beta\in (1/r)T(X)_0$. Then, according to the discussion of (ii) in \S\! \ref{sec:GLUE} and especially Lemma \ref{lem:liftT0}, the obstruction $o(\alpha,L^r)\in H^2(B,\Gm)$ vanishes.
Here, $L$ is the restriction of an ample line bundle $M_\ell$ on a modular
compactification $\check A\subset\km$.  Let us add the hypothesis that $r$ is also coprime 
to $\chi(A_t,\kl_t)$, where $\kl$ is a fixed relative ample line bundle $\kl$ on $X$. We then 
choose  $\ell$  as in Proposition \ref{prop:lbM}, so that its degree $\chi(\check A_t,M_{\ell t})$ is coprime to $r$.\smallskip

Finally, we choose a Lagrangian multi-section $\tilde\PP\subset X$ of the Lagrangian
fibration $X\to \PP$, which always exists as explained in \S\! \ref{sec:rigidify}.
This produces a multi-section $\tilde B\subset A$ of the smooth part $f\colon A\to B$,
say of degree $k$, that has the property that for all $\alpha\in \Br(X)^0$ the restriction
 $\alpha|_{\tilde B}\in H^2(\tilde B,\Gm)$ is trivial. If we eventually
add the condition that $r$ is also coprime to $k$, then Corollary \ref{cor:nosec} allows us to conclude
\begin{equation}\label{eqn:final}
\ind(\alpha)\mid\per(\alpha)^g,
\end{equation}
where $g$ is the relative dimension of $A\to B$ which is nothing but $\dim(X)/2$.
Here, we are implicitly using that period and index of a Brauer class $\alpha$ 
on $X$ and of its restriction to the open part $A\subset X$ coincide.

Combining all the numerical assumptions on $r=\per(\alpha)$ we find that
(\ref{eqn:final}) holds if $r$ is coprime to
$$N_X=\chi(A_t,\kl)\cdot\deg (\tilde\PP/\PP)\cdot |H^3(X,\ZZ)_{\text{ tors}}|\cdot \ind (T(X)_0).$$

At this point it seems unknown whether $H^3(X,\ZZ)_{\text{ tors}}$ could ever be non-trivial. Also, I am not aware of an example with $T(X)_0$
strictly smaller than $T(X)$, so potentially the last factor might just be the index
of $T(X)\subset H^2(X,\ZZ)/\NS(X)$. The first two factors seem hard to avoid with our techniques, but should
of course eventually be superfluous as well.

\subsection{Covering families of Lagrangians}\label{sec:CoveringCubic}
In this last subsection we shall observe that similar techniques also apply to hyperk\"ahler
manifolds which may not admit a Lagrangian fibration but do come with a family of Lagrangian subvarieties. \smallskip

We will study this in the case of the Fano variety of lines. So we consider a smooth cubic fourfold
$Y\subset \PP^5$ and its Fano variety of lines $X\coloneqq F(Y)$, which is a projective
hyperk\"ahler manifold of dimension four, see \cite[Ch.\ 6]{HuyCubics} for a survey and
references.\smallskip

The complete family of hyperplane sections is a family $\ky\to|\ko(1)|\cong\PP^5$
of cubic threefolds
and its relative Fano variety of lines 
\begin{equation}\label{eqn:Fano}
F(\ky)\to |\ko(1)|
\end{equation} can also be viewed as a $\PP^3$-bundle
$F(\ky)\to X=F(Y)$ over $X$.  In particular,
$$\Br(F(\ky))\cong\Br(F(Y)).$$
As $F(Y\cap H)\subset F(X)$ is Lagrangian, this describes a covering family of Lagrangian surfaces and we might regard this as a generalisation of the case of Lagrangian fibrations and apply
the techniques developed earlier. So (\ref{eqn:Fano}) would then play the role of $X\to \PP$ and the analogue of the dual fibration  is provided
by the relative Picard scheme $$\Pic^0(F(\ky)/|\ko(1)|_{\text{sm}})\to |\ko(1)|_{\text{sm}},$$
which  is a principally polarised abelian scheme of relative dimension five.\smallskip

It is convenient to reduce the dimension of the family:
For a generic plane $\PP^2\cong P\subset|\ko(1)|$, one obtains a two-dimensional family of Fano surfaces $F(Y\cap H_t)$, $t\in P$, contained in $X$. The restriction $F(\ky)_P$ to $P$ is in fact birational to $X$
$$\xymatrix{P&\ar[l]_h\ar[r]^-\sim F(\ky)_P&X}$$
and, in particular, pulling back gives an isomorphism $\Br(X)\cong \Br(F(\ky)_P)$.\smallskip

\begin{prop}
Let $Y\subset \PP^5$ be a smooth cubic fourfold and let $X=F(Y)$ be its Fano variety of lines.
Then $$\ind(\alpha)\mid\per(\alpha)^5$$
for every class $\alpha\in \Br(X)$ with $\per(\alpha)$  coprime to the index of  the kernel of the restriction map $T(X)\to H^0(P, R^2h_\ast\ZZ)$ viewed as a sub-lattice of $H^2(X,\ZZ)/\NS(X)$.
\end{prop}

\begin{proof} Note that the result provides a uniform exponent but one that is larger than  predicted by
Conjecture \ref{conj2}.\footnote{As pointed out to me by D.\ Mattei, the result could also
be seen as a consequence of Theorem \ref{thm1} applied to the Lagrangian
fibration described by the family of intermediate Jacobians
of the smooth hyperplane sections $Y\cap H_t$, which compactifies to a smooth
hyperk\"ahler manifold $X\to \PP^5$ of OG10 type, cf.\ \cite{LSV,Sac}. The twists have also 
been studied in the recent paper by Dutta, Mattei, and Shinder \cite{DMS}.} As the main arguments are similar to the ones proving Theorem \ref{thm1},
we will only sketch the main steps. \smallskip

Assume $B\subset P$ is the open part that parametrises smooth hyperplane sections $Y\cap H_t$, $t\in P$ and let $\Pic^0(F(\ky)_B/B)\to B$ be the relative Picard scheme. \smallskip

Since the surfaces $F(Y\cap H)\subset X$ are Lagrangian and $H^3(X,\ZZ)$ is torsion free, every class $\alpha\in \Br(X)\cong\Br(F(\ky)_P)$ restricts trivially to the fibres
$F(Y\cap H_t)$, i.e.\ $\alpha\in \Br_1(F(\ky)_P)$, and, therefore, induces a twist
$\Pic^0_\alpha(F(\ky)_B/B)\to B$. As in \S\! \ref{sec:glueAS}, the  power $L^r$, $r=\per(\alpha)$, of the principal polarisation $L$ on $\Pic^0$ glues to a line bundle $M_\alpha$ on $\Pic_\alpha^0$ if
the obstruction $o(\alpha,L^r)\in H^2(B,\Gm)$ vanishes.\smallskip

Copying the idea in the previous sections, it is suffices to show that $o(\alpha,L^r)$ extends to a class on $\bar B=P\cong \PP^2$. Then we use that $\Br(\PP^2)$ and hence also $o(\alpha,L^r)$ are trivial. The arguments are identical to the ones proving Proposition \ref{prop:lbM}
and the main result in \S\! \ref{sec:VanishingOC}.
\end{proof}

\begin{remark}
(i) As in Remark \ref{rem:2gbis}, it is much easier to establish the bound
$\ind(\alpha)\mid\per(\alpha)^{10}$, at least for all classes $\alpha\in \Br(X)$ with
$\per(\alpha)$ coprime to the degree of an $\alpha$-trivialising multi-section
$\tilde P\subset F(\ky_P)$ of the projection to $P$. An explicit such multi-section
can be obtained as the strict transform  of a generic $F(Y\cap H_0)\subset F(Y)$, with $H_0\not\in P$,
under the birational map $F(\ky_P)\to X=F(Y)$. Its degree is $27$, cf.\ \cite[Thm.\ 6.4.1]{HuyCubics}.
\smallskip

(ii) The Fano variety $F(Y)$ of lines admits a covering family of curves of genus four, namely by the curves of lines passing through a fixed point in $Y$, cf.\ \cite[Ch.\ 5]{HuyCubics}. Applying the techniques of \cite{HuyMa} would then lead to $\ind(\alpha)\mid\per(\alpha)^{8}$. It seems feasible to get the exponent down from $2g=8$ to $g=4$, which would still be far off the conjectured exponent $2$ in this case.
\end{remark}

Similar arguments could be relevant in general. More precisely,
it is expected that every projective hyperk\"ahler manifold comes with
a covering family of Lagrangians $\xymatrix@C=15pt{T&\ar[r]\ar[l]\tilde X&X}$, cf.\  \cite{VoiLag}, which would be potentially useful  for the period-index problem with two obvious drawbacks:
\begin{enumerate}\setlength\itemsep{0.3em}
\item[(i)] Only
 Brauer classes of period coprime to the degree of the covering map $\tilde X\to X$ could be treated.
 \item[(ii)] The bound would look like $\ind(\alpha)\mid\per(\alpha)^{2g}$,
 where $g=h^{1,0}(\tilde X_t)$, which might be much worse than what is predicted
 by Conjectures \ref{conj1} \& \ref{conj2}.
 \end{enumerate}

 Since the parameter space $T$ may have non-trivial Brauer group, the obstruction class
 $o(\alpha,L^r)\in \Br(T)$ discussed in length in \S\! \ref{sec:VanishingOC} might simply not
 be trivial. If it is, the exponent in (ii) could be lowered to $g$.\smallskip
 
A similar observation can be found in \cite[\S\! 2.3]{HuyMa} for K3 surfaces 
using a covering family of elliptic curves which always exist.

\section{Period-index for Hilbert schemes and moduli spaces}
In this section we gather more evidence for Conjecture \ref{conj2}. In particular, we shall prove Theorem \ref{thm2} in \S\! \ref{sec:Hilb}. However, we begin by discussing the
period-index problem for families of curves in \S\! \ref{sec:Jac} building upon the 
discussion in \S\! \ref{sec:PIAbsch}. This has applications to certain moduli spaces of torsion sheaves on K3 surfaces to be discussed in \S\! \ref{sec:ModuliK3curves}, proving Theorem  \ref{thm1} with $N_X=1$. It also provides the opportunity to introduce a technique,
see Lemma \ref{lem:Jaco}, that is then later also used for Hilbert schemes.
We also include short discussions of surface decomposable hyperk\"ahler manifolds \S\!  \ref{sec:surfdecomp} and of the special case of Hilbert schemes of an elliptic K3 surface \S\! \ref{sec:HilbLag}.

\subsection{Period-index for Jacobian fibrations}\label{sec:Jac}
Motivated by the application to Lagrangian fibrations, we have so far considered abelian schemes
$A\to B$ or families that become abelian schemes after generically finite base change, and then twisted the dual abelian scheme $\check A\to B$ by means
of Brauer classes $\alpha\in \Br_1(A)$. We did observe several times that the abelian scheme structure of $A\to B$ was less important than the one of $\check A\to B$ and
will now study a situation where there is no $A\to B$ but only $\check A\to B$. \smallskip

For simplicity, we will consider the easiest and geometrically most meaningful special case: We consider a family of smooth projective curves $f\colon\kc\to B$ of genus $g$ over a quasi-projective variety $B$ over an algebraically closed field of characteristic zero.
The relative Picard scheme$$\check f\colon \Pic(\kc/B)\to B,$$ or rather its identity
component $\check f\colon \Pic^0(\kc/B)\to B,$ will henceforth play the role of the dual abelian
variety $\check A\to B$. The other components will be denoted $\Pic^d(\kc/B)\to B$
and they are naturally torsors for $\Pic^0(\kc/B)\to B$.

Let us now assume that $f$ admits a section. We fix one and denote it by $s\colon B\to \kc$. Then \smallskip

\begin{enumerate}\setlength\itemsep{0.3em}
\item[(i)] All the torsors $\Pic^d(\kc/B)\to B$ are trivial, i.e.\ 
  $\Pic^d(\kc/B)\cong\Pic^0(\kc/B)$ over $B$.
\item[(ii)] There exists a Poincar\'e bundle $\kp\to \Pic(\kc/B)\times_B\kc$.
\end{enumerate}\smallskip

In this case, $\Pic^0(\kc/B)\to B$ comes with a principal polarisation which can
be used to identify it with its own dual. In this sense, it plays the role of $A\to B$ and $\check A\to B$, 
simultaneously.\smallskip

Recall that the map $(x_1,\ldots,x_g)\mapsto \ko(\sum x_i)$ defines a birational
morphism 
\begin{equation}\label{eqn:JacAb}
\kc^{(g)}\to \Pic^g(\kc/B)\cong\Pic^0(\kc/B).
\end{equation}
Here, $\kc^{(d)}$ denotes the relative symmetric product $(\kc\times_B\cdots\times_B\kc)/{\mathfrak S}_d$. The map (\ref{eqn:JacAb}) can be viewed as the relative Albanese
morphism $\kc^{(g)}\to\text{Alb}(\kc^{(g)}/B)\cong\text{Alb}(\kc/B)\cong\Pic^0(\kc/B)$,
where the various identifications all depend on the choice of the section $s$.

Since (\ref{eqn:JacAb}) is a birational projective morphism between smooth varieties, 
it induces an isomorphism
$$\Br(\kc^{(g)})\cong \Br(\Pic^0(\kc/B)).$$
In fact, adapting \cite[\S\! 4]{Biswas} and \cite[Thm.\ 1.2]{IJ} to the relative setting, one finds that 
$$\Br(\kc^{(d)})\cong \Br(\Pic^0(\kc/B))$$
for all $d\geq2$.

We will only deal with Brauer classes that are trivial along the prefixed sections. 
This resembles the choice of a Lagrangian (multi-)section in the hyperk\"ahler setting.
We define $$\Br(\kc^{(d)})_s\coloneqq\ker\left(\Br(\kc^{(d)})\to \Br(s(B))\right)$$
and $$\Br(\Pic^{(d)}(\kc/B))_s\coloneqq\ker\left(\Br(\Pic^{(d)}(\kc/B))\to\Br(s(B))\right),$$
where we denote by $s$ also the sections  $B\to \kc^{(d)}$, $t\mapsto (s(t),\ldots,s(t))$, and $B\to\Pic^d(\kc/B)$, $t\mapsto \ko(ds(t))$, of $\,\kc^{(d)}\to B$ resp.\ $\Pic^d(\kc/B)\to B$.

As the last piece of notation, we introduce the  subgroups 
$$\Br_1(\kc^{(d)})_s\subset \Br(\kc^{(d)})_s\text{ and }\Br_1(\Pic^{d}(\kc/B))_s\subset \Br(\Pic^{d}(\kc/B))_s$$ 
of all elements that are trivial on the geometric fibres.

\begin{lem}\label{lem:Jaco}
The map $\Br(\kc)\to\Br(\kc^{(g)})$, $\alpha\mapsto \alpha^{(g)}\coloneqq\prod_{i=1}^g{\rm pr}_i^\ast\alpha$ induces an isomorphism 
$$\Br(\kc)_s\congpf \Br_1(\kc^{(g)})_s\cong\Br_1(\Pic^g(\kc/B))_s.$$
\end{lem}

\begin{proof} The proof follows ideas in \cite{Biswas, IJ} for curves over algebraically closed fields.
To simplify the notation we shall write $\Pic^d$ for $\Pic^d(\kc/B)$ throughout.

The first step is to show that the pull-back under the natural morphism $$\kc \to \Pic^1\cong \Pic^0,~x\mapsto \ko(x)\mapsto\ko(x-s)$$ induces isomorphisms
$$\Br(\kc)\cong \Br_1(\Pic^0) \text{ and }\Br(\kc)_s\cong \Br_1(\Pic^0)_s.$$
Indeed, the restriction 
map $\Br_1(\Pic^0))\subset \Br(\Pic^0)\to \Br(\kc)$ is compatible with the Leray spectral sequence for the two projections. Since for dimension reasons, the projection
$f\colon \kc\to B$ satisfies $H^0(B,R^2f_\ast\Gm)=0$, we
have $\Br_1(\kc)=\Br(\kc)$. Thus, there is a  commutative diagram 
$$\xymatrix{\Br(B)\ar[d]^=\ar@{^(->}[r]&\Br_1(\Pic^1)\ar[d]\ar[r]&H^1(B,R^1\check f_\ast\Gm)\ar[d]^\cong\ar[r]&H^3(B,\Gm)\ar[d]^=\\
\Br(B)\ar@{^(->}[r]&\Br(\kc)\ar[r]&H^1(B,R^1 f_\ast\Gm)\ar[r]&H^3(B,\Gm),}$$
which is enough to prove the claim.
Note that the injectivity of the pull-back of $\Br(B)$ to the Brauer groups of $\kc$ and $\Pic^1$ 
is a consequences of the existence of a section and the vertical isomorphism follows from the
isomorphism $R^1\check f_\ast\Gm\cong R^1f_\ast\Gm$.

In the second step we use the pull-backs under the morphisms
$$\kc^{(d)}\,\,\hookrightarrow \kc^{(d+1)},~(x_1,\ldots,x_d)\mapsto (x_1,\ldots,x_d,s)$$
and 
$$\Pic^d\,\,\hookrightarrow\Pic^{d+1},~L\mapsto L\otimes\ko(s),$$
to construct the commutative diagram
$$\xymatrix{\Br(\kc)_s&\ar[l]\Br_1(\kc^{(2)})&\ar[l]\cdots&\cdots&\ar[l]\Br_1(\kc^{(g)})_s\\
\Br_1(\Pic^1)_s\ar[u]^-\cong&\ar[l]_\cong\Br_1(\Pic^2)_s\ar[u]&\ar[l]_-\cong\cdots&\cdots&\ar[l]_-\cong\Br_1(\Pic^g)_s\ar[u]^-\cong.}$$
To conclude, observe that the composition in the upper row is split by $\alpha\mapsto \alpha^{(g)}$. Indeed, if $\alpha$ restricts trivially to the section, then the restriction of the $d$-fold exterior product $\alpha^{\times d+1}=\alpha\times\cdots\times \alpha\times\alpha$ to $\kc\times\cdots\times\kc\times \{s\}$ equals the $d$-fold
exterior product $\alpha^{\times d}=\alpha\times\cdots\times \alpha$.
\end{proof}

The important part for us is the surjectivity, i.e.\ any $\beta\in \Br_1(\Pic^0(\kc/B))_s$ is
of the form $\alpha^{(g)}$ for some $\alpha\in \Br(\kc)_s$. Here, we use the identification
$\Pic^0\cong\Pic^g$. Without it, the next result would need to be phrased for classes on $\Pic^g$.

\begin{prop}\label{prop:symcurve}
Assume $\kc\to B$ is a family of smooth projective curves of genus $g$
with a section. Then  for any Brauer class $\beta=\alpha^{(g)}\in \Br_1(\Pic^0(\kc/B))_s$ 
with $\alpha\in \Br(\kc)_s$ one has
$$\ind(\beta)\mid\ind(\alpha)^g.$$
 In particular, if $\ind(\alpha)=\per(\alpha)$,
then $$\ind(\beta)\mid\per(\beta)^g.$$
\end{prop}

\begin{proof}
Assume $\ka$ is an Azumaya algebra on $\kc$ that represents $\alpha$. Then the exterior
product $\kb\coloneqq\ka\boxtimes\cdots\boxtimes\ka$ on $\kc^n\coloneqq\kc\times_B\cdots\times_B\kc$ represents
the class $\alpha^{\times n}$ on $\kc^n$ which descends to the class $\beta=\alpha^{(n)}$ on $\kc^{(n)}$.\smallskip

The algebra $\kb$, which satisfies $\rk(\kb)=\rk(\ka)^n$, comes with a natural linearisation with respect to the action of ${\mathfrak S}_n$. This action is free on the complement
of the big diagonal $\{(x_1,\ldots,x_n)\mid \exists i\ne j: x_i=x_j\}$. Hence, $\kb$ descends to an Azumaya algebra on an open subset  $U\subset\kc^{(n)}$ that represents the restriction $\beta|_U$.
The gcd of the $\rk(\ka)^{1/2}$ over all $\ka$ representing $\ka$ computes $\ind(\alpha)$
and hence $\ind(\beta)$ divides the gcd of all $\rk(\ka)^{n/2}$.
\end{proof} 

In contrast to the discussion in \S\S\! \ref{sec:PIAbsch} \& \ref{sec:HK}, the arguments in this section do not proceed by constructing sheaves on some twist $A_\beta$ and then translating those back to $\beta$-twisted sheaves on the original variety, which here is $\Pic^0(\kc/B)$.
So in this sense, the arguments are much more direct.
\subsection{Period-index for general projective varieties}\label{sec:Improv}
This short section is an interlude
on general projective varieties. We briefly explain how to apply the techniques in \S\!    \ref{sec:proof} to improve upon the result in \cite{HuyMa}. As the
improvement is minor and far from being optimal, with an exponent still depending on the geometry of $X$ and not just its dimension, we will only sketch the main idea.\smallskip

Recall that one of the main results in \cite{HuyMa} asserted that
$\ind(\alpha)\mid\per(\alpha)^{2g}$ for all Brauer classes $\alpha\in \Br(X)$ on an arbitrary projective variety $X$ with $g$ being the genus of a very ample complete intersection curve.
We claim that the techniques developed here now prove the better bound:
$$\ind(\alpha)\mid\per(\alpha)^{g}.$$

Indeed, the proof in \cite{HuyMa} uses  the existence of a family
$$\xymatrix{\kc\ar[d]_-h\,\ar@{^(->}[r]\ar[d]&\bar\kc\ar@{->>}[r]\ar[d]&X\\B\ar@{^(->}[r]&\PP&}$$
with $\PP$ a projective space of dimension $\dim(X)-1$, the projection $\kc\to X$ being
a blow-up, and $\bar\kc\to \PP$ a family of complete intersection curves with its
smooth part $\kc\to B$. Also note that $\bar\kc\to \PP$ comes with a section, the pre-image of
a closed point of $X$ contained in the center of the blow-up, which trivialises all Brauer classes
coming from $X$.\smallskip

The dual fibration $\Pic^0(\kc/B)\to B$ is principally polarised and we pick one such
principal polarisation $L$. If $\alpha\in \Br(X)$ is of period $r$, then $L^r$ glues
to a line bundle on the twist $\Pic_\alpha^0(\kc/B)$ if the obstruction $o(\alpha,L^r)\in \Br(B)$
vanishes. The obstruction class is the image under the boundary map
$H^1(B,R^1h_\ast\mu_r)\to H^2(B,\Gm)$ of a certain natural class
$\alpha'\in H^1(B,R^1h_\ast \mu_r)$ induced by $\alpha$. It then suffices to show that $\alpha'$ extends
to a class in $H^1(\PP,\kk_r)$ with an appropriately defined sheaf $\kk_r$ and that
there exists an extension of $\kk_r$ by $\Gm$ on $\PP$ that extends the given one
on $B$. All this can be achieved and then $o(\alpha,L^r)\in H^2(B,\Gm)$ is the image
of a class under $H^2(\PP,\Gm)\to H^2(B,\Gm)$ and, therefore, trivial, for $\PP$ is a projective
space.\smallskip

The arguments in \S\! \ref{sec:HK} go indeed through and are at places even easier. For example,
working with a covering family of curves makes it unnecessary to restrict to certain
finite index subgroups of $\Br(X)$. Indeed, all classes $\alpha\in \Br(X)$ immediately
lead to twists $\Pic_\alpha^0(\kc/B)$. Hence, all coprimality assumptions on $\per(\alpha)$ can
be avoided.\smallskip


\subsection{Moduli spaces of one-dimensional sheaves on K3 surfaces}\label{sec:ModuliK3curves}
As an application we mention the case of moduli spaces of one-dimensional
sheaves on K3 surfaces.  For more details and references concerning moduli
spaces of sheaves on K3 surfaces we refer to \cite[Ch.\ 10.3]{HuyK3} or \cite[Ch.\ 6]{HuLe}.
\smallskip

We consider a complex projective K3 surface $S$ together with an ample linear system $|H|$ of genus $g$, i.e.\ $(H.H)=2g-2$. Then we let  $$X\coloneqq M(v)$$ be the moduli space of stable sheaves on $S$ with Mukai vector $v=(0,H,1)$. In other words, $M(v)$ generically parametrises line bundles of degree $g$ on curves $C\in |H|$. Stability is taken with respect to a generic polarisation which implies that $X$ is a projective hyperk\"ahler manifold of dimension $2g$ and the support
map describes a Lagrangian fibration $X\to |H|$.\smallskip

Furthermore, there exists a universal family on $S\times X$ which is known to produce
a Hodge isometry $\tilde H(S,\ZZ)\supset v^\perp\cong H^2(X,\ZZ)$ and, as a consequence,
an isomorphism $\Br(S)\cong \Br(X)$ via the description of the Brauer group in terms of the exponential sequence.
\smallskip

The next result is a variant of Theorem \ref{thm1} for Lagrangian fibrations
but without any coprimality condition on the period.

\begin{cor}\label{cor:moduli}
For every class $\alpha\in \Br(X)$ one has
$$\ind(\alpha)\mid\per(\alpha)^g.$$
\end{cor}

\begin{proof} Consider the universal family of smooth curves $\kc\to B\coloneqq|H|_{\text{sm}}$.
Then $\Pic^g(\kc/B)$ is an open subset of $X$ compatible with the projection $X\to |H|$.

We claim that the composition
$$\xymatrix@C=15pt{\Br(S)\ar[r]^-\sim&\Br(X)\ar@{^(->}[r]&\Br(\Pic^g(\kc/B))}$$ of the above isomorphism $\Br(S)\cong \Br(X)$ with the restriction
to the open subset $\Pic^g(\kc/B)\subset X$ is nothing but the map $\alpha\mapsto \alpha^{(g)}$ in Lemma \ref{lem:Jaco}. Then Proposition \ref{prop:symcurve} would immediately imply the assertion. The assertion follows from the discussion in the next section and the fact that
the two isomorphisms $\Br(S)\cong \Br(X)$ and $\Br(S)\cong\Br(S^{[n]})$ are compatible with
the birational identification of $S^{[g]}$ and $X$.\TBC{It is a little circular.}
\end{proof}

As the Hilbert scheme $S^{[g]}$  is in fact birational to $\kc^{(g)}$ and hence to $\Pic^g(\kc/B)$, it  can also be seen as a special case
of Theorem \ref{thm2} or, equivalently, Corollary \ref{cor:thm2} below.
\begin{remark}
The other moduli spaces $M(v)$ with $v=(0,H,d+1-g)$ can be treated similarly by working with
the Abel--Jacobi map $\kc^{(d)}\to\Pic^d(\kc/B)$ with $d\geq g$. However, the exponent becomes worse, namely $\per(\alpha)^d$ instead of $\per(\alpha)^g$.
But one can always reduce to $2g-2\geq d$
\end{remark}

\begin{remark}
According to \cite[Thm.\ 1.2]{HuyMa2}, every Lagrangian fibration of a hyperk\"ahler mani\-fold
of ${\text K3}^{[n]}$-type is birational to some relative twisted 
Picard scheme $\Pic_\gamma^0(\kc/|H|)\to |H|$. Generalising
Proposition \ref{prop:symcurve} to the twisted case would  lead to a stronger version of
Theorem \ref{thm1} for all Lagrangian fibrations of ${\text K3}^{[n]}$-type with no
(or weaker) coprimality assumptions on $\per(\alpha)$.
\end{remark}

\subsection{Hilbert schemes of K3 surfaces}\label{sec:Hilb}
For a K3 surface $S$, we denote by $S^{(n)}$ its symmetric product and by $S^{[n]}$
its Hilbert scheme of subschemes of length $n$. The Hilbert scheme can be viewed as a resolution
of the singular symmetric product via the Hilbert--Chow morphism $\pi\colon S^{[n]}\to S^{(n)}$.\smallskip

We will use the Hodge theoretic description of the Brauer group and so have to work over
the complex number. Thus, from now on,  $S$ denotes a complex projective K3 surface. For fixed $n-1$ pairwise distinct points $x_1,\ldots,x_{n-1}\in S$  the map $x\mapsto \{x_1,\ldots,x_{n-1},x\}$ extends to an embedding
\begin{equation}\label{eqn:ShatHilb}
\hat S\coloneqq \text{Bl}_{\{x_i\}}(S)\,\,\hookrightarrow S^{[n]}.
\end{equation}

\begin{lem}
The pull-back under the embedding {\rm(\ref{eqn:ShatHilb})} induces a primitive embedding
of Hodge structures $$H^2(S^{[n]},\ZZ)\,\hookrightarrow H^2(\hat S,\ZZ)$$
and an isomorphism 
\begin{equation}\label{eqn:BrSHilb2}
\Br(S^{[n]})\cong \Br(\hat S)\cong \Br(S).
\end{equation}

The inverse of this isomorphism is given by 
\begin{equation}\label{eqn:BrSHilb}
\Br(S)\to \Br(S^{(n)})\to \Br(S^{[n]}),~\alpha\mapsto \alpha^{(n)}\mapsto
\alpha^{[n]}\coloneqq\pi^\ast\alpha^{(n).}
\end{equation}
\end{lem}

\begin{proof} According to \cite[\S\! 6]{Beauv}, there exists an isomorphism $H^2(S^{[n]},\ZZ)\cong H^2(S,\ZZ)\oplus\ZZ\delta$. Here, $2\delta$ is the class of the exceptional divisor
of the Hilbert--Chow morphism $\pi\colon S^{[n]}\to S^{(n)}$ and  the inclusion $i\colon H^2(S,\ZZ)\,\hookrightarrow H^2(S^{[n]},\ZZ)$ is given by the condition that $i(a)$ is the pull-back 
under $\pi$ of the class  on the symmetric product $S^{(n)}$ that pulls back to $\sum{\rm pr}_i^\ast a\in H^2(S^n,\ZZ)$ under the quotient map $S^n\to S^{(n)}$.\smallskip

Observe that composing $i$ with the pull-back under (\ref{eqn:ShatHilb}) is simply the
natural embedding $H^2(S,\ZZ)\,\hookrightarrow H^2(\hat S,\ZZ)$. To prove the first assertion, it then suffices to remark that the exceptional class $\delta$ restricts to
the sum $\sum[E_i]$ of the exceptional curves of the blow-up $\hat S\to S$, which is a primitive class in $H^2(\hat S,\ZZ)$.\smallskip

In order to prove (\ref{eqn:BrSHilb2}), we use that $\Br(S)$ is naturally identified with the torsion of
the analytic Brauer group
$H^2(S,\ko_S^\ast)\cong H^2(S, \ko)/\im(H^2(S,\ZZ)\to H^2(S,\ko))$ and similarly for $S^{[n]}$. Then use again the above isomorphism of Hodge structures which in particular induces
$H^{1,1}(S^{[n]},\ZZ)\cong H^{1,1}(S,\ZZ)$.\smallskip

For (\ref{eqn:BrSHilb}), we imitate the arguments  in \S\! \ref{sec:Jac} for families of curves. It suffices to observe that the composition of (\ref{eqn:BrSHilb}) with (\ref{eqn:BrSHilb2}) is the identity.\TBC{CHeck this again.}
\end{proof}

Then Theorem \ref{thm2} is an immediate consequence and we state it again as:

\begin{cor}\label{cor:thm2}
Assume $S$ is a K3 surface. Then for any Brauer class $\beta\in\Br(S^{[n]})$ on the Hilbert scheme $S^{[n]}$ 
one has
$$\ind(\beta)\mid\per(\beta)^n.$$
\end{cor}

\begin{proof} The proof is similar to the one of Proposition \ref{prop:symcurve}.
By virtue of the lemma, we can write any Brauer class $\beta$ on the Hilbert scheme $S^{[n]}$
as $\alpha^{[n]}$ for some Brauer class $\alpha$ on the surface $S$. Using de Jong's solution of the period-index problem for surfaces, see \cite{dJ}, or the alternative proof for complex K3 surfaces in \cite{HuySchr}, we know that  $\alpha$ can be represented by an Azumaya algebra $\ka$ on $S$ of index $\ind(\ka)=\rk(\ka)^{1/2}=\per(\alpha)$.\MD{Is the index for surfaces really given by the minimum or only the gcd? Yes, this is somewhat special for surfaces: Maximal orders are reflexive and hence locally free. See vdBergh's account of de Jong.}\smallskip

Let us denote by $\kb$ the pull-back of the exterior product $\ka\boxtimes \cdots\boxtimes \ka$ on $S^n$  under the blow-up morphism $p\colon{\rm Bl}(S^n)\coloneqq S^{[n]}\times_{S^{(n)}} S^n\to S^n$.
Then $\kb$ is an Azumaya algebra on ${\rm Bl}(S^n)$ which represents the class $p^\ast\alpha^{\times n}$.
Clearly, $\kb$ is naturally linearised with respect to the natural ${\mathfrak S}_n$-action
on ${\rm Bl}(S^n)$. Thus, on the open set where the action is free, i.e.\ the complement of the big diagonal in $S^n$,  the Azumaya algebra $\kb$ descends to an Azumaya algebra on the complement $U\coloneqq S^{[n]}\,\setminus\, E$ of the exceptional locus $E\subset S^{[n]}$ of the Hilbert--Chow morphism. Hence, $\beta|_U$
is represented by an Azumaya algebra of rank $\rk(\ka)^n=\per(\alpha)^{2n}$ and, therefore, 
$\ind(\beta)\mid\per(\alpha)^n$. Since $\per(\alpha)=\per(\beta)$, this proves the result.
\end{proof}

\begin{remark}\label{rem:notfeas}
The proof above only shows $\ind(\beta)\mid\per(\beta)^n$, but not that there actually exists
an Azumaya algebra $\kb$ on $S^{[n]}$ whose index $\ind(\kb)=\rk(\kb)^{1/2}$ divides $\per(\beta)^n$. Examples of Brauer classes on a scheme not representable by an Azumaya
algebra that realises the index at the generic point have been studied before,
e.g.\ by Antieau and Williams \cite{AnW,AW}. \smallskip

The same issue is the major obstacle to apply deformation theory, e.g.\ via hyperholomorphic
bundles and twistor spaces, to generalise our results to arbitrary hyperk\"ahler manifolds.
To even start the process one would need a sheaf of Azumaya algebras or a locally free twisted
sheaf on all of $X$. But even when this can be guaranteed, one would still need to ensure
stability and control over the second Chern class.
\end{remark}

\subsection{Surface decomposable hyperk\"ahler manifolds}\label{sec:surfdecomp}
The Hilbert scheme $S^{[n]}$ of a K3 surface $S$ is rationally dominated
by the product of surfaces $S^n$ and one could study more generally 
the period-index problem on hyperk\"ahler manifolds $X$ that are
rationally dominated by a product of surfaces $\pi\colon S_1\times \cdots \times S_n\to X$.
\smallskip

For Brauer classes $\alpha$ on $X$ that pull-back to product classes $\alpha_1\times\cdots\times\alpha_n$ the solution of the period-index problem for surfaces
can again be applied to prove that $\ind(\alpha)\mid\per(\alpha)^n$ at least for classes of a period coprime
to the degree of $\pi$. However, in general, the pull-back might involve terms coming from $H^1(S_i)\otimes H^1(S_j)$. This does not happen for the Hilbert scheme, as the
pull-back of the two-form on $S^{[n]}$ under the rational map $S^n\to S^{[n]}$ decomposes
into two forms on the factors. This gives an alternative argument to prove Corollary \ref{cor:thm2}
but only for classes with $\per(\alpha)$ coprime to $n!$.\smallskip

 Voisin \cite{VoisinDecom} introduced the notion of surface decomposable varieties.
Instead of dominating rational maps from products of surfaces she allows dominating correspondences
$$\xymatrix@C=15pt{S_1\times\cdots\times S_n&\ar[l]\ar[r] \Gamma&X,}$$
but imposes the condition that the pull-back of $H^{2,0}(X)$ is a linear combination of
pull-backs of classes in $H^{2,0}(S_i)$. Whenever a hyperk\"ahler admits such a surface decomposition, then the period-index problem in the sense
of Conjecture \ref{conj2} can be answered in the affirmative at least for those Brauer classes
with period coprime to the degrees of the two projections and the order of the torsion of $H^3(X,\ZZ)$. \smallskip

This can also be applied to generalised Kummer varieties $K_n(A)$, the second explicit series of examples
of hyperk\"ahler manifolds. 

\begin{cor}
Let $X=K_n(A)$ be a generalised Kummer variety of dimension $2n$ and $\alpha\in\Br(X)$. If $\per(\alpha)$ coprime to $(n+1)!$, then
$$\ind(\alpha)\mid\per(\alpha)^n.$$
\end{cor}

\begin{proof}
The assertion follows from the fact that $K_n(A)$ is rationally dominated by a map
$A^n\to K_n(A)$ of degree $(n+1)!$ which has the necessary property, namely  that $H^{2,0}(K_n(A))$
pulls back to the direct sum $\bigoplus p_i^\ast H^{2,0}(A)$.
\end{proof}

Potentially, a descent argument as in the proof of Corollary \ref{cor:thm2} 
could be used to get rid
of the coprimality condition on $\per(\alpha)$. However, $\ka\boxtimes\cdots\boxtimes\ka$ on $A^n$ is a priori no longer
linearised with respect to the relevant group action, here of ${\mathfrak S}_{n+1}$. But one could try to restrict
$\ka\boxtimes \cdots\boxtimes\ka$ on $A^{n+1}$ to $A^n\,\hookrightarrow A^{n+1}$, $(x_1,\ldots,x_n)\mapsto (x_1,\ldots,x_n,-\sum x_i)$, which then descends to $K_n(A)$ on a dense open subset. This would
eventually prove $\ind(\alpha)\mid\per(\alpha)^{n+1}$ for all Brauer classes
$\alpha\in \Br(K_n(A))$ on the $n$-dimensional hyperk\"ahler manifold $K_n(A)$. We leave the details to the reader.

\subsection{Hilbert schemes of elliptic K3 surfaces}\label{sec:HilbLag}
Assume now that the K3 surface $S$
comes with a genus one fibration $S\to \PP^1$. Then its symmetric product $S^{[n]}$ 
inherits a Lagrangian fibration $S^{[n]}\to (\PP^1)^{[n]}\cong\PP^n$.\smallskip

According to Theorem \ref{thm2} (or Corollary \ref{cor:thm2}),  we have $\ind(\alpha)\mid\per(\alpha)^n$
for all $\alpha\in \Br(S^{[n]})$, while Theorem \ref{thm1} only ensures it for
classes $\alpha\in \Br(S)$ with a period coprime to $$N_S\coloneqq \deg(S/\PP^1)\cdot \text{disc}(\NS(S)).$$
Here, $\deg(S/\PP^1)$ is the minimal degree of a multi-section $C\subset S$ of $S\to \PP^1$.
\smallskip

Indeed, a multi-section $C/\PP^1$ of degree $m$ of $S\to \PP^1$ induces a multi-section
$C^{[n]}/ (\PP^1)^{[n]}$ of degree $m^n$ of $S^{[n]}\to \PP^n$. For $S\to \PP^1$ the minimal degree of a multi-section coincides with the minimal degree of a relative polarisation.
Thus, in the definition of $N_X$ for $X=S^{[n]}$, cf.\ (\ref{eqn:NX}) and \S\! \ref{sec:proof}, the factors $\deg(\kl)$ and $\deg(\tilde \PP/\PP)$
contribute with the same prime factors. Since, $H^3(S^{[n]},\ZZ)$ is torsion free, cf.\ Remark
\ref{rem:TorsionHK}, the only
other contributing factor in the definition of $N_X$ is the index of $$\left(T(S^{[n]})_0 \subset H^2(S^{[n]},\ZZ)/\NS(S^{[n]})\right)\cong\left(T(S)\subset H^2(S,\ZZ)/\NS(S)\phantom{{}^{[n]}\!\!\!\!\!\!}\right)$$ which
is just $\text{disc}(\NS(S))$. Here we use that for a K3 surface $T(S)\to H^0(\PP^1,R^2h_\ast\ZZ)$ is trivial.\MD{Compute stalks of $R^2$ as $H^2$ of fibres and those are indeed
$\bigoplus H^2$ of the components $C_i$ and $T(S)\to H^2(C_i,\ZZ)$ is trivial.}\smallskip

We view this as further evidence, that eventually the condition in Theorem \ref{thm1} of $\per(\alpha)$ being coprime to $N_X$ should be superfluous.



\begin{thebibliography}{99}

\bibitem{Abash} A.\ Abasheva
\em \v{S}afarevi\v{c}--Tate groups of holomorphic Lagrangian fibrations II. \em 
\href{https://arxiv.org/abs/2407.09178}{arXiv:2407.09178}.

\bibitem{AbRo} A.\ Abasheva, V.\ Rogov
\em \v{S}afarevi\v{c}--Tate groups of holomorphic Lagrangian fibrations. \em 
\href{https://arxiv.org/abs/2112.10921}{arXiv:2112.10921}.


\bibitem{AnW}
B.\ Antieau, B.\  Williams
\em Unramified division algebras do not always contain Azumaya maximal orders. \em
Invent.\ Math.\ 197 (2014), 47--56.

\bibitem{AW}
B.\ Antieau, B.\  Williams
\em  The prime divisors of the period and index of a Brauer class. \em
J.\ Pure Appl.\  Algebra 219 (2015), 2218--2224.








\bibitem{Artin} M.\ Artin
\em Brauer--Severi varieties. \em 
In: Brauer groups in ring theory and algebraic geometry (Wilrijk, 1981)
Lecture Notes in Math., 917 Springer-Verlag, Berlin-New York (1982), 194--210.





\bibitem{BS} B.\ Bakker, C.\ Schnell
\em A Hodge-theoretic proof of Hwang's theorem on base manifolds of Lagrangian fibrations. \em 
\href{https://arxiv.org/abs/2311.08977}{arXiv:2311.08977}.


\bibitem{Beauv} A.\ Beauville
\em Vari\'et\'es k\"ahleriennes dont la premi\`ere classe de Chern est nulle.   \em
J.\ Diff.\ Geom.\ 18 (1983), 755--782.

\bibitem{Bi}
C.\ Birkenhake, H.\ Lange
\em The dual polarization of an abelian variety. \em
Arch.\ Math.\ 73 (1999), 380--389.

\bibitem{Biswas} I.\ Biswas, A.\ Dhillon, J.\ Hurtubise
\em Brauer groups of Quot schemes. \em
Mich.\ Math.\ J.\ 64 (2015), 493--508.



%

\bibitem{BLR} S.\ Bosch, W.\ L\"utkebohmert, M.\ Raynaud
\em N\'eron Models. \em Ergebnisse der Math.\ und ihrer Grenzgebiete, 3.\ Folge, Band 21 (1990).



\bibitem{Clark} P.\ Clark
\em The period-index problem in WC-groups II: abelian varieties. \em
\href{https://arxiv.org/abs/0406135}{arXiv:math/0406135}.



\bibitem{CTBr}
J.-L.\ Colliot-Th\'el\`ene
\em Die Brauersche Gruppe; ihre Verallgemeinerungen und Anwendungen in der arithmetischen Geometrie. \em Vortragsnotizen, Brauer  Tagung, Stuttgart, 22.-24. M\"arz 2001. 
\href{https://arxiv.org/abs/2311.02437}{arXiv:2311.02437}

\bibitem{CTGab} J.-L. Colliot-Th\'el\`ene
\em Exposant et indice d'alg\`ebres simples centrales non ramifi\'ees. \em
 With an appendix by Ofer Gabber, Enseign.\ Math.\ 48 (2002), 127--146.


\bibitem{CTS} J.-L.\ Colliot-Th\'el\`ene, A.\ Skorobogatov
\em The Brauer--Grothendieck group. \em
 Ergebnisse der Mathematik und ihrer Grenzgebiete. 3.\ Folge, volume 71. 2021.

\bibitem{DCaRaSa} M.\ De Cataldo, A.\ Rapagnetta, G.\ Sacc\`a
\em The Hodge numbers of O'Grady 10 via Ng\^o strings. \em
J.\ Math.\ Pures et Appl.\ 156 (2021), 125--178.

\bibitem{dJ} A.\ J.\ de Jong
\em The period-index problem for Brauer groups of an algebraic surfaces.
\em Duke Math.\ J.\ 123 (2004), 71--94.


\bibitem{dJS}
A.\ J.\ de Jong, J.\ Starr
\em Almost proper GIT-stacks and discriminant avoidance. \em
Doc.\ Math.\ 15 (2010), 957--972. 
 available at \href{https://www.math.columbia.edu/~dejong/papers/5-torsor.pdf}{Longer version}.

\bibitem{dJP} A.\ J.\ de Jong, A.\ Perry
\em The period-index problem and Hodge theory.
\em \href{https://arxiv.org/abs/2212.12971}{arXiv:2212.12971}.






\bibitem{DIKM} Y.\ Dutta, E.\ Izadi, L.\ Kamenova, L.\  Marquand
\em Some density results for hyperk\"ahler manifolds. \em
\href{https://arxiv.org/abs/2403.04868}{arXiv:2403.04868}.

\bibitem{DMS} Y.\ Dutta, D.\ Mattei, E.\ Shinder
\em Twists of intermediate Jacobian fibrations. \em
\href{https://arxiv.org/abs/2411.01953}{arXiv:2411.01953}.



\bibitem{GS}
P.\ Gille, T.\ Szamuely
\em Central simple algebras and Galois cohomology. \em
volume 101 of Cambridge Studies in Advanced Mathematics. 
Cambridge University Press, 2006. 

\bibitem{GHJ} M.\ Gross, D.\ Huybrechts, D.\ Joyce
\em Calabi--Yau manifolds and related geometries. \em
Springer (2002).

\bibitem{Brauer} A.\ Grothendieck
\em Le groupe de Brauer I--III.   \em
In: Dix Expos\'es  sur la cohomologie des sch\'emas. volume 3 of Adv.\  Stud.\  in Pure Maths. North Holland Amsterdam (1968).

\bibitem{SGA1} A.\ Grothendieck
\em Rev\^etements \'etales et groupe fondamental (SGA1). \em
\'Edition recompos\'ee et annot\'ee. Do\-cu\-ments Math\'ematiques 3. SMF (2003).


\bibitem{HoPe} J.\ Hotchkiss,  A.\ Perry
\em The period-index conjecture for abelian threefolds and Donaldson--Thomas theory. 
\em \href{https://arxiv.org/abs/2405.03315}{arXiv:2405.03315}.






\bibitem{HuyK3} D.\ Huybrechts
\em Lectures on K3 surfaces. \em
volume 158 of Cambridge Studies in Advanced Mathematics. 
Cambridge University Press, Cambridge, 2016.

\bibitem{HuyCubics} D.\ Huybrechts
\em The geometry of cubic hypersurfaces. \em
volume 206 of Cambridge Studies in Advanced Mathematics. 
Cambridge University Press, Cambridge, 2023.

\bibitem{HuLe} D.\ Huybrechts, M.\ Lehn
\em The geometry of moduli spaces of sheaves. \em
Cambridge Mathematical Library. Cambridge University Press. (2010)


\bibitem{HuyMa} D.\ Huybrechts, D.\ Mattei
\em Splitting unramified Brauer classes by abelian torsors and the period-index problem. \em 
\href{https://arxiv.org/abs/2310.04029}{arXiv:2310.04029}.

\bibitem{HuyMa2} D.\ Huybrechts, D.\ Mattei
 \em The special Brauer group and twisted Picard varieties. \em
 \href{https://arxiv.org/abs/2310.04032}{arXiv:2310.04032}.
 
\bibitem{HuyMau} D.\ Huybrechts, M.\ Mauri
\em Lagrangian fibrations. \em
Milan J.\ Math.\ 90 (2022), 459--483.

 \bibitem{HuySchr} D.\ Huybrechts, S.\ Schr\"oer,
\em  The Brauer group of analytic K3 surfaces. \em
IMRN.\ 50 (2003), 2687--2698. 


 \bibitem{HuyXu} D.\ Huybrechts, C.\ Xu
 \em Lagrangian fibrations of hyperk\"ahler fourfolds. \em
 J.\  IJM 21 (2022), 921--932. 
 
 \bibitem{Hwang} J.-M.\ Hwang
\em Base manifolds for fibrations of projective irreducible symplectic manifolds. \em 
Invent.\ Math.\ 174 (2008), 625--644.


\bibitem{IJ} J.\ Iyera, R.\ Joshua
\em Brauer groups of schemes associated to symmetric powers of smooth projective curves in arbitrary characteristics. \em
J.\  of Pure and Appl.\ Alg.\ 224 (2020), 1009--1022.


\bibitem{KaMe} S.\ Kapfer, G.\ Menet
\em Integral cohomology of the generalized Kummer fourfold. \em
Algebr.\ Geom.\ 5 (2018), 523--567. \& \em Corrigendum \em  Algebr.\ Geom.\ 7 (2020), 482--485.

\bibitem{KLB} D.\ Krashen, M.\  Lieblich
\em Index reduction for Brauer classes via stable sheaves. \em
IMRN 2008 (2008) 31 pp.

\bibitem{LSV} R.\ Laza, G.\  Sacc\`a, C.\ Voisin
\em A hyper-K\"ahler compactification of the intermediate Jacobian fibration associated with a cubic 4-fold. \em
Acta Math.\ 218 (2017), 55--135.

\bibitem{LiTo} Y.\ Li, V.\ Tosatti
\em Special K\"ahler geometry and holomorphic Lagrangian fibrations. \em 
 \href{https://arxiv.org/abs/2308.10553}{arXiv:2308.10553}.

\bibitem{Li} S.\ Lichtenbaum
\em The period-index problem for elliptic curves. \em
American J.\ Math.\  90 (1968), 1209--1223.


\bibitem{LiebDuke} M.\ Lieblich
\em Moduli of twisted  sheaves. \em
Duke Math.\ J.\ 138 (2007), 23--118.


\bibitem{LiebAnn} M.\ Lieblich
\em The period-index problem for fields of transcendence degree 2. \em
Ann.\  Math.\ 182 (2015), 391--427.

\bibitem{Lin}  H.-Y.\ Lin
\em Lagrangian constant cycle subvarieties in Lagrangian fibrations. \em
IMRN, 2020, No.\  1 (2020), 14--24.

\bibitem{MarkInt} E.\ Markman
\em Integral generators for the cohomology ring of moduli spaces of sheaves over Poisson surfaces. \em
Adv.\ Math.\ 208 (2007), 622--646.

\bibitem{MarkMeh} E.\ Markman, S.\ Mehrotra
\em Hilbert schemes of K3 surfaces are dense in moduli. \em Math.\ Nachr.\ 290 (2017), 876--884.



\bibitem{MoRaSa} G.\ Mongardi, A.\ Rapagnetta, G.\ Sacc\`a
\em The Hodge diamond of O'Grady's six-dimensional example. \em
Comp.\  Math.\ 154 (2018), 984--1013. 

\bibitem{Mukai} S.\ Mukai
\em Duality between $D(X)$ and $D(\hat X)$ with its application to Picard sheaves. \em
Nagoya Math.\ J.\ 81 (1981), 153--175.

\bibitem{MumfordTheta} D.\ Mumford
\em On the equations defining abelian varieties. I. \em
Invent.\ Math.\ 1 (1966), 287--354.

\bibitem{MumAV} D.\ Mumford
\em Abelian varieties. \em
Oxford University Press  (1970).


\bibitem{Ogg} A.\ Ogg
\em Cohomology of abelian varieties over function fields. \em 
Ann.\ Math.\ 76 (1962), 185--212.

\bibitem{OG1} K.\ O'Grady
 \em Desingularized moduli spaces of sheaves on a K3. \em
J.\ Reine Angew.\ Math.\ 512 (1999), 49--117.

\bibitem{OG2} K.\ O'Grady
\em  A new six-dimensional irreducible symplectic manifold. \em
JAG 12 (2003), 435--505.




\bibitem{Sac} G.\ Sacc\`a
\em Birational geometry of the intermediate Jacobian fibration of a cubic fourfold. \em
With an appendix by Claire Voisin
Geom.\ Topol.\   27 (2023), 1479--1538.

\bibitem{Sha} I.\ \v{S}afarevi\v{c}
\em Principal homogeneous spaces defined over a function field. \em
Trudy Mat.\ Inst.\ Steklov. 64 (1961), 316--346.

\bibitem{Simpson}
C.\ Simpson
\em Moduli of representations of the fundamental group of a smooth projective variety I. \em Publ.\ math.\ IHES 79 (1994), 47--129.




\bibitem{Tot} B.\ Totaro
\em The integral cohomology of the Hilbert scheme of points on a surface. \em
Forum of Mathematics, Sigma 8 (2020) e40.


\bibitem{VoisinHodge} C.\ Voisin
\em Th\'eorie de Hodge et g\'eom\'etrie alg\'ebrique complexe. \em Cours sp\'ecialis\'es 10 SMF (2002).


\bibitem{VoisinDecom} C.\ Voisin
\em Triangle varieties and surface decomposition of hyper-K\"ahler manifolds. \em
In: Recent Developments in Algebraic Geometry: To Miles Reid for his 70th Birthday, LMS
Lecture Note Series 478, Cambridge University Press (2022), 315--356.


\bibitem{VoiLag} C.\ Voisin
\em On the Lefschetz standard conjecture for Lagrangian covered hyper-K\"ahler varieties. \em Adv.\ Math.\  396 (2022), Paper No. 108108, 29 pp.

\end{thebibliography}
\end{document}